\documentclass[11pt]{article}

\usepackage[T1]{fontenc}
\usepackage{mathtools, amssymb, amsthm}
\usepackage[margin=2.5cm]{geometry}
\usepackage[british]{babel}
\usepackage{microtype}
\usepackage{enumitem}
\usepackage{xcolor}
\usepackage[hidelinks]{hyperref}

\title{Entropic optimal transport need not select a zero-temperature limit}
\author{Maja Gw\'o\'zd\'z\\
ETH Z\"{u}rich\\
\texttt{mgwozdz@ethz.ch}}
\date{}

\theoremstyle{plain}
\newtheorem{theorem}{Theorem}[section]
\newtheorem{lemma}[theorem]{Lemma}
\newtheorem{proposition}[theorem]{Proposition}
\newtheorem{corollary}[theorem]{Corollary}
\newtheorem{claim}[theorem]{Claim}

\theoremstyle{definition}
\newtheorem{definition}[theorem]{Definition}
\theoremstyle{remark}
\newtheorem{remark}[theorem]{Remark}
\numberwithin{equation}{section}

\newcommand{\N}{\mathbb{N}}
\newcommand{\R}{\mathbb{R}}
\newcommand{\Id}{\mathrm{Id}}
\newcommand{\pr}{\mathrm{pr}}
\newcommand{\dd}{\,\mathrm{d}}

\DeclareMathOperator{\Clust}{Clust}

\providecommand{\onevec}{\mathbf{1}}
\providecommand{\ind}[1]{\onevec_{\{#1\}}}
\providecommand{\diag}{\operatorname{diag}}
\renewcommand{\epsilon}{\ensuremath\varepsilon}
\renewcommand{\phi}{\ensuremath{\varphi}}

\newcommand{\Ent}{\mathcal H}

\begin{document}

\maketitle

\begin{abstract}
We construct a compact metric space with an atomless probability measure and a
bounded Lipschitz cost for which the entropic optimal-transport minimisers have
no zero-temperature weak limit. More precisely, \(P_\varepsilon\) does not converge as \(\varepsilon\downarrow0\). In the example, every unregularised minimiser is singular with respect to \(\mu\otimes\mu\), so that the entropy on the optimal face is identically \(+\infty\). We describe the cluster set by
\[
\Clust(P_\varepsilon)=\{P_w:w\in\mathcal W\},
\]
where \(P_w\) is the mixture of the two zero-cost graph couplings with weight \(w\), and where \(\mathcal W\subset[0,1]\) is a non-degenerate compact interval. We then compute two explicit points \(w^-<w^+\) in this interval. This shows that compactness, atomlessness, and Lipschitz regularity of the cost do not imply zero-temperature convergence. We also present a compactness theorem for the general problem. If
\(C\in L^1(\mu\otimes\nu)\) is continuous and bounded from below on Polish
spaces, then the zero-temperature cluster set is a nonempty weakly compact
connected subset of the optimal face. In the proof, we apply the cluster-point theorem
of Bernton, Ghosal, and Nutz and the continuity of \(\varepsilon\mapsto\pi_\varepsilon\). Finally, we give local and exterior first-order criteria for full convergence and cluster membership. We show that nonconvergence is possible, but only through a connected continuum of optimal plans.
\end{abstract}



\section{Introduction}\label{sec:introduction}
Let \(\varepsilon>0\). In entropic optimal transport, we replace the standard Monge--Kantorovich problem by the minimisation of
\begin{equation}\label{eq:Feps-def}
F_\varepsilon(\pi):=\int_{X\times Y} C\,\dd\pi+\varepsilon\,\mathrm{KL}(\pi\|\mu\otimes\nu),
\qquad \pi\in\Pi(\mu,\nu),
\end{equation}
over the set of couplings of \(\mu\) and \(\nu\). Under the assumptions below,
there exists a unique minimiser, which we denote by \(\pi_\varepsilon\). In other terms,
\(\pi_\varepsilon\) is the static Schr\"odinger bridge associated with
\[
R_\varepsilon\propto e^{-C/\varepsilon}(\mu\otimes\nu).
\]
For the related background, consult \cite{Leonard2014,DiMarinoGerolin2020}. The limit
\(\varepsilon\downarrow0\) is a selection problem on the optimal face of the
unregularised transport problem. The first natural question is whether \(\pi_\varepsilon\) converges. If it does, one may ask which optimal plan is selected.

In this work, we prove that convergence can fail even in a compact atomless
bounded-Lipschitz setting. We first establish the general topological restriction: the zero-temperature cluster set is always a nonempty compact
connected subset of the optimal face. The optimality of cluster points follows
from \cite[Proposition~3.2]{BGN}, while connectedness follows from weak
compactness and continuity of the entropic curve. We then construct a dyadic
Cantor-group model in which this connected set is nontrivial. The mechanism is
an alternation of scales, analogous to chaotic temperature dependence for
zero-temperature Gibbs states \cite{vanEnterRuszel2007,ChazottesHochman2010}. Here, the oscillating object is the Schr\"odinger projection under fixed marginal constraints.

\subsection{Main results}
Let
\[
\mathrm{Opt}(\mu,\nu\mid C)
:=
\arg\min_{\pi\in\Pi(\mu,\nu)}\int_{X\times Y} C\,\dd\pi.
\]
For a family \((z_\varepsilon)_{\varepsilon>0}\) in a metrisable topological
space, we write
\[
\Clust(z_\varepsilon):=\{z:\exists\,\varepsilon_n\downarrow0\text{ such that }z_{\varepsilon_n}\to z\}.
\]

\begin{theorem}[Connected zero-temperature cluster sets]\label{thm:intro-general}
Let \(X\) and \(Y\) be Polish spaces, let \(\mu\in\mathcal P(X)\), \(\nu\in\mathcal P(Y)\), and let \(C:X\times Y\to\mathbb R\) be continuous, bounded from below, and in
\(L^1(\mu\otimes\nu)\). We use \(\pi_\varepsilon\in\Pi(\mu,\nu)\) to denote the
entropic minimiser. In this case, the weak cluster set of \((\pi_\varepsilon)_{\varepsilon>0}\), as \(\varepsilon\downarrow0\), is a
nonempty compact connected subset of the set of optimal plans for \(C\).
\end{theorem}

Our first result is a compactness theorem. In second theorem, we show that this
compactness statement cannot be strengthened to convergence.

\begin{theorem}[Compact Lipschitz nonconvergence]\label{thm:intro-counterexample}
There exist a compact metric space \(X\), an atomless probability measure
\(\mu\in\mathcal P(X)\), and a bounded Lipschitz cost \(C:X\times X\to[0,\infty)\) for which the entropic minimisers \(P_\varepsilon\in\Pi(\mu,\mu)\) do not converge weakly as \(\varepsilon\downarrow0\). Every optimal plan is singular with respect to
\(\mu\otimes\mu\). In the dyadic Cantor-group construction,
\[
\Clust(P_\varepsilon)=
\{P_w:w\in\mathcal W\},
\]
where
\[
P_w=wP_0+(1-w)P_{t_0}
\]
and \(P_0,P_{t_0}\) are the two zero-cost graph couplings from
\eqref{eq:cantor-graph-couplings}. Moreover, \(\mathcal W=\Clust(w_\varepsilon)\), with \(w_\varepsilon\) defined in \eqref{eq:cantor-what-def}, is a non-degenerate compact interval in \([0,1]\). For \(a\in(0,1)\) and \(\beta\in(0,\infty)\setminus\{1\}\), this interval
contains the two distinct weights
\[
w^\pm
=
\frac{\sqrt{(1+e^{-1})(1+e^{-\beta})}}
{\sqrt{(1+e^{-1})(1+e^{-\beta})}
+\sqrt{(1+e^{-(1\pm a)})(1+e^{-\beta(1\pm a)})}},
\qquad w^-<w^+ .
\]
\end{theorem}

We do not identify the whole interval \(\mathcal W\), as it suffices to compute two points \(w^-<w^+\) and establish connectedness, which already gives nonconvergence and a non-degenerate continuum of subsequential limits.  There is also a useful comparison with the lower-semicontinuous diagonal example
\[
c(x,y)=\mathbf 1_{\{x\ne y\}},\qquad
\mu=\nu=\mathrm{Unif}[0,1],
\]
given in \cite[Remark~3.3]{BGN} and discussed in \cite[Example~5.2 and the following discussion]{NutzWiesel2022}. In that example, the diagonal is \(\mu\otimes\mu\)-null, which implies that changing the cost on the diagonal does not change the positive-temperature problem. In our example, the cost is continuous and bounded Lipschitz. We obtain the obstruction by positive off-graph values near the zero set, such that
these values are still visible after the Schr\"odinger projection onto the fixed
marginal constraints.

Bernton--Ghosal--Nutz \cite{BGN} discuss several positive convergence regimes and note that convergence had already been conjectured in a general setting. In Theorem~\ref{thm:intro-counterexample}, we disprove this general principle under compactness, atomless marginals, and bounded Lipschitz regularity of the cost. We show that it suffices to make the Schr\"odinger projection oscillate to obtain the result.

\subsection{Related work}\label{subsec:related-work}
The entropy-to-transport limit and the connection with the Schr\"odinger problem are studied in L\'eonard~\cite{Leonard2012MK,Leonard2014}. Nutz and
Wiesel~\cite{NutzWiesel2022} prove compactness and convergence results for
Schr\"odinger potentials, and in one dimension with distance cost, Di Marino and
Louet~\cite{DiMarinoLouet2018} determine the selected plan under additional
assumptions on the marginals. See also Cominetti and San Mart\'in
\cite{CominettiSanMartin1994} for the finite-dimensional analogue and selecting a point of the optimal face. In linear programming, Weed~\cite{Weed2018} established direct asymptotic estimates for the entropic penalty in linear programming.

Other relevant works study convergence of entropic schemes
\cite{CarlierDuvalPeyreSchmitzer2017}, convergence rates \cite{CarlierPegonTamanini2023,Pal2024}, stability
\cite{GhosalNutzBernton2022}, regularity of the Schr\"odinger cost \cite{Clerc2022}, displacement smoothness \cite{CarlierChizatLaborde2024}, and entropic approximation of \(\infty\)-optimal transport problems \cite{BrizziCarlierDePascale2024}. Certain recent selection results impose additional geometry. For instance, Ley~\cite{Ley2025} analyses the line with distance cost, while Nutz and Zhong~\cite{NutzZhong2026} study \(c(x,y)=\|x-y\|\) in dimensions \(d>1\). Aryan and Ghosal~\cite{AryanGhosal2025} prove related ray-wise entropy limits for the Euclidean distance cost, and Yun~\cite{Yun2025} obtains the Gaussian selection result. On compact metric spaces with Lipschitz costs, Mengue~\cite{Mengue2025} developed a zero-temperature pressure formulation. This result recovers the Monge--Kantorovich problem and its duality in the limit, but it does not constitute a selection theorem for the entropic minimisers in \eqref{eq:Feps-def}.

\subsection{Idea of the proof}
We first prove the compactness theorem. In the proof, we use weak compactness of the coupling set, continuity of \(\varepsilon\mapsto\pi_\varepsilon\), and
optimality of zero-temperature cluster points. More precisely, it represents
\(\Clust(\pi_\varepsilon)\) as a nested intersection of compact connected sets. We then reduce the counterexample to a scalar oscillation. We set \(H:=\{s\in G:s_1=0\}\), and prove that the diagonally \(H\)-invariant part of the zero-cost face is exactly
\[
\{\,wP_0+(1-w)P_{t_0}:0\le w\le1\,\}
\]
and has singular zero-cost couplings. At positive temperature, the
Schr\"odinger system then reduces to a \(2\times2\) Sinkhorn scaling problem. The two different subsequential limits in the branch weight follow from the scales \(\alpha_n=e^{-2^n}\) and the alternating factors \(L_n\). Once we have identified all weak cluster points with the points on the segment, we apply the connectedness theorem to turn those two limits into a non-degenerate interval.

The paper is organised as follows. In Section~\ref{sec:entropic-prelim}, we recall the compactness and lower-semicontinuity facts for entropic minimisers. In Section~\ref{sec:general-theory}, we present the connectedness theorem, and in Sections~\ref{sec:cantor-model}--%
\ref{sec:cantor-oscillation-main}, we analyse the dyadic Cantor-group construction. In Appendix~\ref{sec:selection-criteria}, we prove the local and exterior variational criteria. Finally, in Appendix~\ref{sec:appendix-sinkhorn}, we study the scaling algebra, and in Appendix~\ref{app:H-invariant-face}, we classify the diagonally \(H\)-invariant zero-cost couplings.

\section{Preliminaries on entropic minimisers}\label{sec:entropic-prelim}
We first recall certain facts on the compactness, lower-semicontinuity, and existence relevant in the general theory and in the Cantor model. For a Polish space \(S\), we write \(\mathcal P(S)\) for the Borel probability measures on \(S\). In this section and Section~\ref{sec:general-theory}, \(X\) and \(Y\) are Polish spaces with their Borel \(\sigma\)-fields, \(\mu\in\mathcal P(X)\), \(\nu\in\mathcal P(Y)\), and
\[
\Lambda:=\mu\otimes\nu\in\mathcal P(X\times Y).
\]
We assume that
\[
C:X\times Y\to\mathbb R
\]
is continuous, bounded from below, and integrable with respect to \(\Lambda\),
that is,
\[
\int_{X\times Y}|C|\,\dd\Lambda<\infty.
\]
We write
\[
\Pi(\mu,\nu):=\{\pi\in\mathcal P(X\times Y):(\pr_X)_\#\pi=\mu,\ (\pr_Y)_\#\pi=\nu\},
\]
and let \(\Rightarrow\) denote weak convergence on \(\mathcal P(X\times Y)\). For \(\pi\in\Pi(\mu,\nu)\), we define
\[
J(\pi):=\int_{X\times Y} C\,\dd\pi\in(-\infty,+\infty],
\]
and, for \(M\in\mathcal P(X\times Y)\) and \(P\in\mathcal P(X\times Y)\), we set
\[
\mathrm{KL}(P\|M):=
\begin{cases}
\displaystyle \int_{X\times Y}\log\!\Bigl(\frac{\dd P}{\dd M}\Bigr)\,\dd P,
& P\ll M,\\[0.8em]
+\infty, & P\not\ll M.
\end{cases}
\]
In particular,
\[
\Ent(\pi):=\mathrm{KL}(\pi\|\Lambda),
\qquad \pi\in\Pi(\mu,\nu).
\]
Let
\[
C_0:=\inf_{\pi\in\Pi(\mu,\nu)}J(\pi)\in\mathbb R.
\]
We also set
\[
\mathrm{Opt}(\mu,\nu\mid C):=\{\pi\in\Pi(\mu,\nu):J(\pi)=C_0\},
\]
and write
\[
\Delta(\pi):=J(\pi)-C_0\in[0,\infty].
\]
For \(\varepsilon>0\), let
\[
F_\varepsilon(\pi):=J(\pi)+\varepsilon \Ent(\pi),
\qquad \pi\in\Pi(\mu,\nu).
\]

\begin{proposition}[Standard facts on the entropic minimiser]\label{prop:standard-eot-facts}
Using the assumptions of this section, the coupling set \(\Pi(\mu,\nu)\) is weakly compact. Moreover, the functionals \(J\) and \(\Ent\) are weakly lower semicontinuous on \(\Pi(\mu,\nu)\), and \(\mathrm{Opt}(\mu,\nu\mid C)\) is nonempty. For every \(\varepsilon>0\), \(F_\varepsilon\) has a unique minimiser \(\pi_\varepsilon\in\Pi(\mu,\nu)\), and this minimiser satisfies
\[
F_\varepsilon(\pi_\varepsilon)<\infty,
\qquad
\Ent(\pi_\varepsilon)<\infty,
\qquad
\pi_\varepsilon\ll\Lambda.
\]
\end{proposition}

\begin{proof}
By Prokhorov's theorem and the closedness of the marginal constraints, it follows that
\(\Pi(\mu,\nu)\) is weakly compact. Since \(C\) is bounded from below and lower
semicontinuous, by Portmanteau's theorem, we obtain the weak lower semicontinuity of
\(J\). The weak lower semicontinuity of \(\Ent\) is due to \cite{Csiszar1975}. Moreover, \(\Lambda\in\Pi(\mu,\nu)\), \(J(\Lambda)<\infty\), and \(\Ent(\Lambda)=0\). It follows that \(F_\varepsilon\) is not identically \(+\infty\). If \(m:=\inf_{X\times Y}C\), then
\[
m\le C_0\le J(\Lambda)<\infty,
\]
so the unregularised value is finite. By the direct method, we obtain a minimiser of
\(F_\varepsilon\) and also an optimal transport plan. Notice that every entropic minimiser has finite \(F_\varepsilon\)-value and finite entropy, so it is absolutely continuous with respect to \(\Lambda\). Finally, any two minimisers have densities with respect to \(\Lambda\), and relative entropy is strictly convex on densities, while \(J\) is affine on its finite effective domain. This proves uniqueness.
\end{proof}

\section{Connected zero-temperature cluster sets}\label{sec:general-theory}
We apply two results of Bernton--Ghosal--Nutz \cite[Propositions~2.2 and~3.2]{BGN}. In their notation, for a nonnegative continuous cost \(c\), the first result states that, under the finite-value condition
\[
\exists\,\pi\in\Pi(\mu,\nu):\qquad
\int c\,\dd\pi+\mathrm{KL}(\pi\|\mu\otimes\nu)<\infty,
\]
the entropic minimiser is the unique \((c,\varepsilon)\)-cyclically invariant
coupling. The second result implies that every weak cluster point, as
\(\varepsilon\downarrow0\), of these \((c,\varepsilon)\)-cyclically invariant
couplings has \(c\)-cyclically monotone support. This means that it is \(c\)-optimal
whenever the unregularised value is finite. In our context, we apply the result to
\(\widetilde C=C-\inf_{X\times Y}C\). The finite-value condition is satisfied, because
\(\Lambda\in\Pi(\mu,\nu)\), \(\Ent(\Lambda)=0\), and
\[
\int \widetilde C\,\dd\Lambda<\infty.
\]
If we pass from \(C\) to \(\widetilde C\), we do not change either the entropic minimisers
or the unregularised minimisers. \cite{NutzWiesel2022} also give the related compactness and convergence results for Schr\"odinger potentials under continuous integrable costs on Polish spaces. Let
\[
\Omega(\mu,\nu\mid C):=\Clust(\pi_\varepsilon)
=\{\pi\in\Pi(\mu,\nu):\exists\,\varepsilon_k\downarrow0
\text{ with }\pi_{\varepsilon_k}\Rightarrow\pi\}
\]
denote the zero-temperature cluster set.

\subsection{Connectedness of the zero-temperature cluster set}\label{subsec:general-connectedness}
\begin{theorem}[Connectedness of the zero-temperature cluster set]\label{thm:selection-connected}
Using the assumptions of Section~\ref{sec:entropic-prelim}, the set
\[
\Omega(\mu,\nu\mid C)
\]
is a nonempty weakly compact connected subset of \(\mathrm{Opt}(\mu,\nu\mid C)\).
\end{theorem}

\begin{lemma}[Nested-closure of the zero-temperature cluster set]
\label{lem:omega-nested-closure}
It holds that
\[
\Omega(\mu,\nu\mid C)
=
\bigcap_{m\ge1}\overline{\{\pi_\varepsilon:0<\varepsilon<1/m\}}^{\,\mathrm{weak}}.
\]
\end{lemma}

\begin{proof}
The inclusion \(\subset\) follows immediately. For the converse, let \(d\) be a metric that
generates the weak topology on the compact metrisable space \(\Pi(\mu,\nu)\).
We also assume that \(\pi\) belongs to the right-hand side. This implies that \(\pi\in
\overline{\{\pi_\varepsilon:0<\varepsilon<\delta\}}^{\,d}\) for every
\(\delta>0\). Indeed, choose \(m\) such that \(1/m<\delta\). We obtain
\[
\{\pi_\varepsilon:0<\varepsilon<1/m\}\subset
\{\pi_\varepsilon:0<\varepsilon<\delta\},
\]
so
\[
\overline{\{\pi_\varepsilon:0<\varepsilon<1/m\}}^{\,d}
\subset
\overline{\{\pi_\varepsilon:0<\varepsilon<\delta\}}^{\,d}.
\]
We now choose \(\varepsilon_1\in(0,1)\) with
\(d(\pi_{\varepsilon_1},\pi)<1\), and for \(m\ge2\),
\[
\varepsilon_m\in\Bigl(0,\min\{1/m,\varepsilon_{m-1}/2\}\Bigr)
\]
such that \(d(\pi_{\varepsilon_m},\pi)<1/m\). It follows that
\(\varepsilon_m\downarrow0\) and \(\pi_{\varepsilon_m}\Rightarrow\pi\).
\end{proof}

\begin{proposition}[Cluster points are optimal transports]\label{prop:cluster-optimal-Polish}
Every weak cluster point of \((\pi_\varepsilon)_{\varepsilon>0}\), as
\(\varepsilon\downarrow0\), is an optimal transport plan. In other words,
\[
\Omega(\mu,\nu\mid C)\subset\mathrm{Opt}(\mu,\nu\mid C).
\]
\end{proposition}

\begin{proof}
Let \(m:=\inf_{X\times Y}C>-\infty\), and set \(\widetilde C:=C-m\). Notice that the
entropic minimisers for \(C\) and \(\widetilde C\) coincide. Moreover,
\[
\int \widetilde C\,\dd\Lambda<\infty,
\qquad
\Ent(\Lambda)=0.
\]
Let \(\pi_{\varepsilon_n}\Rightarrow\bar\pi\) with \(\varepsilon_n\downarrow0\). Since the marginal constraints are weakly closed, \(\bar\pi\in\Pi(\mu,\nu)\). The previous equation verifies the condition with finite values in \cite[Proposition~2.2]{BGN}. It follows that \((\pi_\varepsilon)_{\varepsilon>0}\) is the family of
\((\widetilde C,\varepsilon)\)-cyclically invariant couplings considered in
\cite[Proposition~3.2]{BGN}. We obtain that \(\bar\pi\) has \(\widetilde C\)-cyclically monotone support, and so \(\bar\pi\) is \(\widetilde C\)-optimal. Notice that adding a constant to the cost does not change the minimisers, which implies that \(\bar\pi\) is \(C\)-optimal.
\end{proof}

\begin{remark}
Earlier convergence results, which include optimality of cluster points in small-noise limits, go back to L\'eonard~\cite{Leonard2012MK}. We use here the BGN cluster-point theorem because it applies to the fixed static product reference and to our
continuous \(L^1(\mu\otimes\nu)\) context.
\end{remark}

\begin{proposition}[Continuity of the entropic trajectory]\label{prop:eps-continuity-Polish}
The entropic trajectory
\[
(0,\infty)\ni\varepsilon\longmapsto\pi_\varepsilon\in\Pi(\mu,\nu)
\]
is weakly continuous.
\end{proposition}

\begin{proof}
We fix \(\varepsilon_0>0\), and let \(\varepsilon_n\to\varepsilon_0\). By weak
compactness of \(\Pi(\mu,\nu)\), it is enough to consider a weakly convergent
subsequence of \((\pi_{\varepsilon_n})\) and identify its limit. Let \(\pi_{\varepsilon_{n_k}}\Rightarrow\bar\pi\). Since \(\Lambda\in\Pi(\mu,\nu)\) and \(\Ent(\Lambda)=0\),
\[
F_{\varepsilon_{n_k}}(\pi_{\varepsilon_{n_k}})
\le F_{\varepsilon_{n_k}}(\Lambda)=J(\Lambda)<\infty.
\]
Let \(m:=\inf_{X\times Y}C>-\infty\). We obtain
\[
F_{\varepsilon_{n_k}}(\pi_{\varepsilon_{n_k}})
=
J(\pi_{\varepsilon_{n_k}})+\varepsilon_{n_k}\Ent(\pi_{\varepsilon_{n_k}})
\ge m+\varepsilon_{n_k}\Ent(\pi_{\varepsilon_{n_k}}),
\]
so
\[
\varepsilon_{n_k}\Ent(\pi_{\varepsilon_{n_k}})\le J(\Lambda)-m.
\]
Since \(\varepsilon_{n_k}\to\varepsilon_0>0\), the entropies are uniformly bounded.

By lower semicontinuity of \(J\) and \(\Ent\),
\[
J(\bar\pi)\le \liminf_{k\to\infty}J(\pi_{\varepsilon_{n_k}}),
\qquad
\Ent(\bar\pi)\le \liminf_{k\to\infty}\Ent(\pi_{\varepsilon_{n_k}}).
\]
From \(\sup_k \Ent(\pi_{\varepsilon_{n_k}})<\infty\) and
\(\varepsilon_{n_k}\to\varepsilon_0\), we deduce
\[
\lim_{k\to\infty}\bigl|(\varepsilon_{n_k}-\varepsilon_0)\Ent(\pi_{\varepsilon_{n_k}})\bigr|=0.
\]
Therefore,
\[
\liminf_{k\to\infty}\bigl(J(\pi_{\varepsilon_{n_k}})+\varepsilon_0\Ent(\pi_{\varepsilon_{n_k}})\bigr)
=
\liminf_{k\to\infty}F_{\varepsilon_{n_k}}(\pi_{\varepsilon_{n_k}}),
\]
and lower semicontinuity of \(J\) and \(\Ent\) gives
\[
F_{\varepsilon_0}(\bar\pi)\le \liminf_{k\to\infty}F_{\varepsilon_{n_k}}(\pi_{\varepsilon_{n_k}}).
\]

On the other hand, we notice that \(\pi_{\varepsilon_0}\) minimises \(F_{\varepsilon_0}\), and \(F_{\varepsilon_0}(\Lambda)<\infty\), which implies
\(\Ent(\pi_{\varepsilon_0})<\infty\), and
\[
\varepsilon\longmapsto F_\varepsilon(\pi_{\varepsilon_0})=J(\pi_{\varepsilon_0})+\varepsilon \Ent(\pi_{\varepsilon_0})
\]
is continuous at \(\varepsilon_0\). By minimality of \(\pi_{\varepsilon_{n_k}}\),
\[
\limsup_{k\to\infty}F_{\varepsilon_{n_k}}(\pi_{\varepsilon_{n_k}})
\le
\limsup_{k\to\infty}F_{\varepsilon_{n_k}}(\pi_{\varepsilon_0})
=
F_{\varepsilon_0}(\pi_{\varepsilon_0}).
\]
We now combine these estimates to obtain \(F_{\varepsilon_0}(\bar\pi)\le F_{\varepsilon_0}(\pi_{\varepsilon_0})\). By uniqueness of the minimiser of \(F_{\varepsilon_0}\), this gives \(\bar\pi=\pi_{\varepsilon_0}\).
\end{proof}

\begin{lemma}[Nested intersection of continua]\label{lem:nested-continua}
Let \((K_m)_{m\ge1}\) be a nested sequence of nonempty compact connected
subsets of a Hausdorff space. It follows that \(\bigcap_{m\ge1}K_m\) is nonempty, compact, and connected.
\end{lemma}

\begin{proof}
Nonemptiness and compactness follow immediately from the finite intersection property in the compact set \(K_1\). We set
\[
K:=\bigcap_{m\ge1}K_m.
\]
Suppose that \(K\) is disconnected. Given that \(K\) is compact and the ambient space
is Hausdorff, the two closed members of a separation of \(K\) have disjoint open
neighbourhoods in the ambient space. This implies the existence of disjoint open sets \(U,V\) such that
\[
K\subset U\cup V,
\qquad
K\cap U\neq\varnothing,
\qquad
K\cap V\neq\varnothing.
\]
Let us set \(F:=(U\cup V)^c\), which is a closed set. Suppose no \(K_m\) is contained in \(U\cup V\), then every \(K_m\cap F\) is nonempty. The sets \(K_m\cap F\) are nested and compact, so their intersection is nonempty. However, this contradicts
\[
K\cap F=\varnothing.
\]
It must hold that some \(K_m\subset U\cup V\). Since \(K_m\) is connected and \(U,V\) are disjoint open sets, either \(K_m\subset U\) or \(K_m\subset V\). This
contradicts the fact that \(K\subset K_m\) meets both \(U\) and \(V\). We conclude that
\(K\) is connected.
\end{proof}

\begin{proof}[Proof of Theorem~\ref{thm:selection-connected}]
For \(m\ge1\), we set
\[
K_m:=\overline{\{\pi_\varepsilon:0<\varepsilon<1/m\}}^{\,\mathrm{weak}}\subset\Pi(\mu,\nu).
\]
Notice that each \(K_m\) is nonempty and weakly compact, because \(\Pi(\mu,\nu)\) is weakly compact. By Proposition~\ref{prop:eps-continuity-Polish}, the map
\[
(0,\infty)\ni\varepsilon\longmapsto\pi_\varepsilon
\]
is weakly continuous. It follows that \(\{\pi_\varepsilon:0<\varepsilon<1/m\}\) is the continuous image of the connected interval \((0,1/m)\), so it is also connected, and its weak closure \(K_m\) is connected as well. The family \((K_m)_{m\ge1}\) is nested,
so by Lemma~\ref{lem:nested-continua},
\[
\bigcap_{m\ge1}K_m
\]
is nonempty, weakly compact, and connected. By Lemma~\ref{lem:omega-nested-closure},
\[
\Omega(\mu,\nu\mid C)=\bigcap_{m\ge1}K_m.
\]
Moreover, by Proposition~\ref{prop:cluster-optimal-Polish}, every element of
\(\Omega(\mu,\nu\mid C)\) is optimal. This proves that \(\Omega(\mu,\nu\mid C)\) is nonempty, weakly compact, connected, and contained in \(\mathrm{Opt}(\mu,\nu\mid C)\).
\end{proof}

\begin{proof}[Proof of Theorem~\ref{thm:intro-general}]
This follows directly from Theorem~\ref{thm:selection-connected}.
\end{proof}

\section{The dyadic Cantor-group model and its optimal face}\label{sec:cantor-model}

\subsection{Model definition}
We now define the compact model. The zero-cost set will be the union of two
graphs: the diagonal and one translate of the diagonal. Away from these graphs,
we separate the positive values of the cost by exponential scales. By a parity-dependent factor, we then create an oscillation along \(\varepsilon_n=e^{-2^n}\). Finally, the first-coordinate factor \(\gamma\) prevents the raw Gibbs kernel from already having the prescribed marginals.

We fix \(a\in(0,1)\) and \(\beta\in(0,\infty)\setminus\{1\}\). We do not need the condition
\(\beta\neq1\) for the scale separation itself. We impose it simply to avoid the degenerate case where the Gibbs reference is already a coupling, and to force a nontrivial Schr\"odinger projection. Let
\[
G:=\{0,1\}^{\mathbb N}
\]
with coordinatewise addition modulo \(2\), and endow it with the dyadic
ultrametric
\[
d(s,t):=
\begin{cases}
0, & s=t,\\
2^{-n(s,t)}, & s\neq t,
\end{cases}
\qquad
n(s,t):=\min\{k\ge1:\ s_k\neq t_k\}.
\]
Let \(\mu=\big(\tfrac12\delta_0+\tfrac12\delta_1\big)^{\otimes\N}\), which is
the Haar probability on \(G\), and set \(t_0:=(1,0,0,\dots)\). We define
\[
\alpha_n:=e^{-2^n},\qquad
L_n:=
\begin{cases}
1+a,& n \text{ even},\\
1-a,& n \text{ odd},
\end{cases}
\qquad n\ge2,
\]
and, for \(t=(t_k)_{k\ge1}\in G\), we set
\[
m(t):=\min\{k\ge2:\ t_k=1\}\in\{2,3,\dots\}\cup\{\infty\},
\]
with the convention that \(m(t)=\infty\) if and only if \(t_k=0\) for all
\(k\ge2\). We set
\begin{equation}\label{eq:cantor-Wdef}
W(t):=
\begin{cases}
0,& t=0 \text{ or } t=t_0,\\
\alpha_{m(t)},& t_1=0,\ t\neq0,\\
\alpha_{m(t)}L_{m(t)},& t_1=1,\ t\neq t_0,
\end{cases}
\end{equation}
\[
\gamma(x):=\ind{x_1=0}+\beta\,\ind{x_1=1}.
\]

\begin{definition}[Dyadic Cantor-group model]\label{def:cantor-explicit-model}
Let \(X=Y:=G\), \(\nu:=\mu\), and \(\Lambda:=\mu\otimes\mu\). The cost is
\begin{equation}\label{eq:cantor-cost}
C(x,y):=\gamma(x)\,W(y-x),\qquad (x,y)\in G^2.
\end{equation}
\end{definition}
We use the notation
\[
D(x,y):=y-x,\qquad
U_0:=\{t\in G:t_1=0\},\qquad
U_1:=\{t\in G:t_1=1\}.
\]
For \(n\ge2\), we set
\[
\begin{aligned}
A_n^0&:=\{t\in G:\ t_1=0,\ t_2=\cdots=t_{n-1}=0,\ t_n=1\},\\
A_n^1&:=\{t\in G:\ t_1=1,\ t_2=\cdots=t_{n-1}=0,\ t_n=1\},
\end{aligned}
\]
and
\[
Z:=\{(x,y)\in G^2:\ y-x\in\{0,t_0\}\}.
\]
By Lemma~\ref{lem:cantor-continuity}, \(Z\) is the zero-cost set of \(C\). We also use the two graph couplings
\begin{equation}\label{eq:cantor-graph-couplings}
P_0:=(\Id,\Id)_\#\mu,\qquad
P_{t_0}:=(\Id,T_{t_0})_\#\mu,\qquad
T_{t_0}(x):=x+t_0.
\end{equation}

\subsection{Basic properties and zero-cost set}\label{sec:cantor-def-zero}
Since \(G\) is a compact topological group and \(\mu\) is Haar probability, translations preserve \(\mu\). In particular, for every integrable \(f:G\times G\to\R\),
\[
\int_G\int_G f(x,y-x)\,\mu(\dd y)\,\mu(\dd x)
=\int_G\int_G f(x,t)\,\mu(\dd t)\,\mu(\dd x).
\]

\begin{lemma}[Continuity and zero set of \(W\) and \(C\)]\label{lem:cantor-continuity}
The function $W$ is continuous and bounded on $G$. Moreover, $W(t)=0$ if and
only if \(t\in\{0,t_0\}\). In particular, the cost \(C\) is continuous and bounded on \(G^2\), and
\[
C(x,y)=0\quad\Longleftrightarrow\quad y-x\in\{0,t_0\}.
\]
\end{lemma}

\begin{proof}
We first notice that the cylinder sets \(A_n^0,A_n^1\) are clopen in \(G\). On \(A_n^0\), we have \(m(t)=n\) and \(W\equiv\alpha_n\). On \(A_n^1\), \(m(t)=n\) and
\(W\equiv \alpha_nL_n\) hold. We infer that \(W\) is locally constant on \(G\setminus\{0,t_0\}\), and is continuous there. It remains to consider \(0\) and \(t_0\). Let \(t^{(j)}\to0\), and note that for every fixed \(N\), the first \(N\) coordinates of \(t^{(j)}\) eventually vanish. For all large \(j\), either \(t^{(j)}=0\) or \(m(t^{(j)})\ge N+1\). Therefore,
\[
0\le W(t^{(j)})\le \alpha_{N+1}.
\]
If we first let \(j\to\infty\) and then \(N\to\infty\), we get \(W(t^{(j)})\to0=W(0)\). The proof at \(t_0\) is identical. If \(t^{(j)}\to t_0\), then for every \(N\), eventually, \(t^{(j)}_1=1\) and
\(t^{(j)}_2=\cdots=t^{(j)}_N=0\). It follows that either \(t^{(j)}=t_0\) or \(m(t^{(j)})\ge N+1\), so
\[
0\le W(t^{(j)})\le (1+a)\alpha_{N+1}
\]
for all large \(j\). Again, \(W(t^{(j)})\to0=W(t_0)\). Boundedness follows directly from \(0\le W\le (1+a)\alpha_2\). By definition, \(W(t)=0\) if and only if \(m(t)=\infty\), or \(t_k=0\) for all \(k\ge2\), that is, \(t\in\{0,t_0\}\).

Given that \(\gamma\) is constant on the two clopen cylinders \(\{x_1=0\}\) and
\(\{x_1=1\}\), it is continuous and bounded on \(G\). The map \((x,y)\mapsto y-x\) is continuous on \(G^2\). It follows that \((x,y)\mapsto W(y-x)\) is continuous on \(G^2\), and \(C(x,y)=\gamma(x)W(y-x)\) is continuous and bounded. Since \(\gamma(x)>0\) for
all \(x\), we also have
\[
C(x,y)=0 \quad\Longleftrightarrow\quad W(y-x)=0 \quad\Longleftrightarrow\quad y-x\in\{0,t_0\}.
\]
\end{proof}

\begin{lemma}[Lipschitz regularity]\label{lem:cantor-lipschitz}
If we endow \(G^2\) with the maximum product metric induced by the dyadic ultrametric
\(d\) on \(G\), then \(C\) is Lipschitz for this metric.
\end{lemma}

\begin{proof}
We first prove that \(W\) is Lipschitz on \(G\). To this end, let \(s\ne t\), and let
\(N\) be the first coordinate at which they differ. If \(N=1\), then
\[
|W(s)-W(t)|\le (1+a)\alpha_2
\le 2(1+a)\alpha_2\, d(s,t).
\]
If \(N\ge2\) and some coordinate among \(2,\ldots,N-1\) is nonzero, then
\(m(s)=m(t)<N\) and \(s_1=t_1\), so \(W(s)=W(t)\). Otherwise, the coordinates
\(2,\ldots,N-1\) of both points vanish. In that case, both \(W(s)\) and \(W(t)\) are in the interval \([0,(1+a)\alpha_N]\), so
\[
|W(s)-W(t)|\le (1+a)\alpha_N
=(1+a)2^N e^{-2^N}\,d(s,t).
\]
Since \(\sup_{N\ge2}2^Ne^{-2^N}<\infty\), this proves that \(W\) is Lipschitz.

The difference map \(D(x,y)=y-x\) is Lipschitz from \(G^2\), with the maximum
product metric, to \(G\), because the dyadic metric is translation invariant and
ultrametric. Notice that the factor \(\gamma\) is also Lipschitz. Indeed, it is constant on the two first-coordinate cylinders, and points in different such cylinders have
distance \(1/2\). Since \(W\) and \(\gamma\) are bounded Lipschitz functions and
\(C(x,y)=\gamma(x)W(D(x,y))\), the cost \(C\) is Lipschitz on \(G^2\).
\end{proof}

\subsection{The optimal face}\label{sec:cantor-opt-face}
We shall also use the fact that \(\mu\) is atomless. Indeed, for every \(x\in G\), the
cylinder fixing the first \(n\) coordinates of \(x\) has mass \(2^{-n}\), and
the intersection of these cylinders is \(\{x\}\), so \(\mu(\{x\})=0\).

\begin{proposition}[Optimal face in the Cantor model]\label{prop:cantor-opt-face}
For the dyadic Cantor-group cost,
\[
\inf_{P\in\Pi(\mu,\mu)}\int C\,\dd P = 0,
\qquad
\mathrm{Opt}(\mu,\mu\mid C)=\{P\in\Pi(\mu,\mu):P(Z)=1\}.
\]
In particular, \(P_0\) and \(P_{t_0}\) are distinct optimal couplings, and every
optimal plan is singular with respect to \(\Lambda=\mu\otimes\mu\).
\end{proposition}

\begin{proof}
By the translation invariance of Haar measure, we obtain \(P_0,P_{t_0}\in\Pi(\mu,\mu)\). On \(\operatorname{supp} P_0\), we have \(y-x=0\), while on
\(\operatorname{supp} P_{t_0}\), \(y-x=t_0\) holds. By Lemma~\ref{lem:cantor-continuity},
\[
\int C\,\dd P_0=\int C\,\dd P_{t_0}=0.
\]
Since \(C\ge 0\), the optimal value is \(0\).

If \(P\in\Pi(\mu,\mu)\) satisfies \(P(Z)=1\), then Lemma~\ref{lem:cantor-continuity}
gives \(C=0\) \(P\)-a.s., hence \(P\) is optimal. On the other hand, if
\(P\in\mathrm{Opt}(\mu,\mu\mid C)\), then \(\int C\,\dd P=0\). Since \(C\ge0\), we get \(C=0\) \(P\)-a.s., so \(P(Z)=1\).

It remains to prove singularity. Note that the set \(Z\) is Borel, and each vertical
section \(Z_x=\{x,x+t_0\}\) has two points. Since \(\mu\) is atomless,
\[
\Lambda(Z)=\int_G \mu(Z_x)\,\mu(\dd x)=0.
\]
Every optimal plan \(P\) satisfies \(P(Z)=1\), while \(\Lambda(Z)=0\), which implies
\(P\perp\Lambda\). By the definition of relative entropy,
\[
\mathrm{KL}(P\|\Lambda)=+\infty
\]
for every \(P\in \mathrm{Opt}(\mu,\mu\mid C)\).
\end{proof}

In the argument below, we consider the diagonally \(H\)-invariant part. Lemma~\ref{lem:cantor-h-classification} shows that this part is precisely
\[
\{\,wP_0+(1-w)P_{t_0}:0\le w\le1\,\}.
\]

\section{Schr\"odinger projection in the dyadic model}\label{sec:cantor-schrodinger-main}
For the rest of the paper, we set
\[
H:=\{s\in G:s_1=0\},\qquad
G_0:=H,\qquad
G_1:=t_0+H.
\]
In the Cantor model, the Gibbs reference measure
\(R_\varepsilon\propto e^{-C/\varepsilon}\Lambda\) does not have marginal
\(\mu\). Therefore, the entropic minimiser is a nontrivial Schr\"odinger
projection. Notice that the diagonal \(H\)-symmetry reduces this projection to a
\(2\times2\) Sinkhorn scaling. Lemma~\ref{lem:cantor-R-not-coupling} shows this clearly when \(\beta\neq1\). In this section, we use \(P_\varepsilon\) to denote the minimiser in the Cantor model, while we use \(\pi_\varepsilon\) for the general theory.

\subsection{Bounded-cost facts for the Cantor Schr\"odinger projection}
\label{sec:cantor-schrodinger-prelim}
We analyse two bounded-cost facts for the Cantor-model argument. We write
\(X\) and \(Y\) for compact metric spaces, \(\mu\in\mathcal P(X)\), \(\nu\in\mathcal P(Y)\), \(C:X\times Y\to[0,\infty)\) is bounded and continuous,
\(\Lambda:=\mu\otimes\nu\), and \(F_\varepsilon\) is the entropic functional
defined in \eqref{eq:Feps-def}. For \(\varepsilon>0\), we set
\[
c_\varepsilon:=\int_{X\times Y} e^{-C/\varepsilon}\,\dd\Lambda\in(0,\infty),
\qquad
\frac{\dd R_\varepsilon}{\dd \Lambda}(x,y):=c_\varepsilon^{-1}e^{-C(x,y)/\varepsilon}.
\]

\paragraph{Entropy identity}\label{sec:prelim-entropy}
\begin{lemma}[Entropy identity]\label{lem:entropy-identity}
Let us assume that \(C\) is bounded. For every \(P\in\Pi(\mu,\nu)\),
\[
F_\varepsilon(P)=\varepsilon\,\mathrm{KL}(P\|R_\varepsilon)-\varepsilon\log c_\varepsilon,
\]
with the convention \(+\infty=+\infty-\varepsilon\log c_\varepsilon\).
\end{lemma}

\begin{proof}
We set
\[
r_\varepsilon:=\frac{\dd R_\varepsilon}{\dd\Lambda}=c_\varepsilon^{-1}e^{-C/\varepsilon}.
\]
Since \(C\) is bounded, \(r_\varepsilon\) is bounded above and below by positive
constants. It follows that \(R_\varepsilon\sim\Lambda\), and \(P\ll\Lambda\) if and only
if \(P\ll R_\varepsilon\). In the case when this absolute continuity fails, both sides of the
identity are \(+\infty\). Let us assume \(P\ll\Lambda\). We then have
\[
\log\frac{\dd P}{\dd R_\varepsilon}
=\log\frac{\dd P}{\dd\Lambda}+\frac{C}{\varepsilon}+\log c_\varepsilon.
\]
Once we integrate with respect to \(P\), we obtain
\[
\mathrm{KL}(P\|R_\varepsilon)
=\mathrm{KL}(P\|\Lambda)+\frac{1}{\varepsilon}\int C\,\dd P+\log c_\varepsilon.
\]
By rearrangement, we have
\[
F_\varepsilon(P)=\varepsilon\,\mathrm{KL}(P\|R_\varepsilon)-\varepsilon\log c_\varepsilon.
\]
\end{proof}

\paragraph{A sufficient criterion for entropic optimality}\label{sec:prelim-factor}
\begin{lemma}[A sufficient criterion for entropic optimality]\label{lem:schrodinger-sufficient}
Let us fix \(\varepsilon>0\) and assume that \(C\) is bounded. Let \(P\in\Pi(\mu,\nu)\) satisfy
\[
\frac{\dd P}{\dd R_\varepsilon}(x,y)=a(x)b(y),
\]
where \(a:X\to(0,\infty)\) and \(b:Y\to(0,\infty)\) are measurable and bounded above and below by positive constants. It follows that \(P\) is the unique minimiser of
\(Q\mapsto \mathrm{KL}(Q\|R_\varepsilon)\) over \(\Pi(\mu,\nu)\). In particular, \(P\) is the unique minimiser of \(F_\varepsilon\) over \(\Pi(\mu,\nu)\).
\end{lemma}

\begin{proof}
Since \(a\) and \(b\) are bounded above and below away from \(0\), the density
\(\frac{\dd P}{\dd R_\varepsilon}=a\otimes b\) is bounded above and below away
from \(0\). In particular, \(P\sim R_\varepsilon\). Let \(Q\in\Pi(\mu,\nu)\).
If \(Q\not\ll R_\varepsilon\), then \(\mathrm{KL}(Q\|R_\varepsilon)=+\infty\), and the minimality claim follows immediately. We may assume that \(Q\ll R_\varepsilon\) and
\(\mathrm{KL}(Q\|R_\varepsilon)<\infty\). This implies \(Q\ll P\), and since
\[
\log\!\Bigl(\frac{\dd P}{\dd R_\varepsilon}\Bigr)=\log a(x)+\log b(y)
\]
is bounded, by the chain rule, we obtain
\[
\mathrm{KL}(Q\|R_\varepsilon)-\mathrm{KL}(P\|R_\varepsilon)
=\mathrm{KL}(Q\|P)+\int_{X\times Y}\log\!\Big(\frac{\dd P}{\dd R_\varepsilon}\Big)\,\dd(Q-P).
\]
Note that the last integral vanishes because \(Q\) and \(P\) have the same marginals.
We then have
\[
\mathrm{KL}(Q\|R_\varepsilon)-\mathrm{KL}(P\|R_\varepsilon)=\mathrm{KL}(Q\|P)\ge0.
\]
This proves that \(P\) minimises \(\mathrm{KL}(\cdot\|R_\varepsilon)\). If
equality holds, then \(\mathrm{KL}(Q\|P)=0\), so \(Q=P\), and the minimiser is
unique. By Lemma~\ref{lem:entropy-identity}, the minimisers of \(\mathrm{KL}(\cdot\|R_\varepsilon)\) are exactly the minimisers of \(F_\varepsilon\).
\end{proof}

\subsection{An averaging lemma on \texorpdfstring{$H$}{H}}

\begin{lemma}[Haar averaging on $H$]\label{lem:cantor-Haar-averaging}
Let $H=\{s\in G:s_1=0\}$ and let $h$ denote the Haar probability measure on $H$
(or $h=2\,\mu|_{H}$). Let \(q:G\to\mathbb R\) be Borel measurable and integrable with respect to \(\mu\), and define
\[
(\mathsf A q)(x):=\int_H q(x+s)\,h(\dd s),\qquad x\in G.
\]
The following statements hold:
\begin{enumerate}[label=(\roman*),leftmargin=*]
\item $\mathsf A q$ is Borel measurable and belongs to $L^1(\mu)$;
\item $\mathsf A q$ is \textit{pointwise $H$-invariant}:
\[
(\mathsf A q)(x+r)=(\mathsf A q)(x)\qquad\text{for all }x\in G,\ r\in H;
\]
\item if for every $r\in H$, $q(x+r)=q(x)$ for $\mu$-a.e.\ $x$, then
\[
\mathsf A q=q\qquad \mu\text{-a.e.}
\]
\end{enumerate}
\end{lemma}

\begin{proof}
For \(i\in\{0,1\}\), let \(\mu_i:=2\,\mu|_{G_i}\), the normalised restriction
of \(\mu\) to the coset \(G_i\). If \(x\in G_i\), then \(x+H=G_i\), and the translation map
\[
T_x:H\to G_i,\qquad s\mapsto x+s,
\]
pushes \(h\) forward to \(\mu_i\). We obtain
\begin{equation}\label{eq:Aq-coset-formula}
(\mathsf A q)(x)=\int_H q(x+s)\,h(\dd s)=\int_{G_i} q(u)\,\mu_i(\dd u),
\qquad x\in G_i.
\end{equation}
Since \(q\in L^1(\mu)\) and \(\mu_i\ll\mu\), we infer that the right-hand side is finite.
Therefore, \(\mathsf A q\) is well-defined everywhere. By formula
\eqref{eq:Aq-coset-formula}, we know that \(\mathsf A q\) is constant on each
coset \(G_i\). In particular, it is Borel measurable and pointwise \(H\)-invariant.
Moreover,
\[
\int_G |\mathsf A q|\,\dd\mu
=\frac12\sum_{i=0}^1 \int_{G_i} |\mathsf A q(x)|\,\mu_i(\dd x)
\le \frac12\sum_{i=0}^1 \int_{G_i} |q(u)|\,\mu_i(\dd u)
= \int_G |q|\,\dd\mu,
\]
so \(\mathsf A q\in L^1(\mu)\). This proves \textit{(i)} and \textit{(ii)}.

Let us assume now that \(q(x+r)=q(x)\) for \(\mu\)-a.e.\ \(x\), for every \(r\in H\).
We define
\[
\Delta(x):=\int_H |q(x+s)-q(x)|\,h(\dd s)\ge 0.
\]
For each fixed \(s\in H\),
\[
\int_G |q(x+s)-q(x)|\,\mu(\dd x)=0.
\]
After integration over \(s\), by Fubini, we obtain
\[
\int_G \Delta(x)\,\mu(\dd x)=0,
\]
so \(\Delta(x)=0\) for \(\mu\)-a.e.\ \(x\). For such \(x\), we have
\(q(x+s)=q(x)\) for \(h\)-a.e.\ \(s\), so
\[
(\mathsf A q)(x)=\int_H q(x+s)\,h(\dd s)=q(x).
\]
Finally, \(\mathsf A q=q\) \(\mu\)-a.e., as required.
\end{proof}

\subsection{Reduction to a \texorpdfstring{$2\times2$}{2x2} Sinkhorn scaling}\label{sec:cantor-schrodinger}

We fix \(\varepsilon>0\). Since \(C(x,y)=\gamma(x)W(y-x)\), we have
\begin{equation}\label{eq:cantor-Reps}
\frac{\dd R_\varepsilon}{\dd\Lambda}(x,y)
=c_\varepsilon^{-1}\exp\!\Big(-\frac{\gamma(x)W(y-x)}{\varepsilon}\Big),
\end{equation}
where
\[
c_\varepsilon:=\int_{G^2}e^{-C/\varepsilon}\,\dd\Lambda\in(0,1].
\]
Since \(C\) is bounded, \(R_\varepsilon\sim\Lambda\) follows.

\begin{lemma}[The Gibbs reference is not a coupling]\label{lem:cantor-R-not-coupling}
We assume $\beta\neq 1$, so \(R_\varepsilon\) is not a coupling of \(\mu\) with
itself, for every $\varepsilon>0$.
\end{lemma}

\begin{proof}
Let \(A\subset G\) be Borel. By Haar invariance and the change of variables
\(t=y-x\), we obtain
\[
R_\varepsilon(A\times G)
=\frac{1}{c_\varepsilon}\int_A\int_G \exp\!\Big(-\frac{\gamma(x)W(t)}{\varepsilon}\Big)\,\mu(\dd t)\,\mu(\dd x)
=\frac{1}{c_\varepsilon}\int_A k_\varepsilon(x)\,\mu(\dd x),
\]
where
\[
k_\varepsilon(x):=
\begin{cases}
\int_G e^{-W(t)/\varepsilon}\,\mu(\dd t), & x_1=0,\\[0.4em]
\int_G e^{-\beta W(t)/\varepsilon}\,\mu(\dd t), & x_1=1.
\end{cases}
\]
It follows that the first marginal of \(R_\varepsilon\) is \(c_\varepsilon^{-1}k_\varepsilon\,\mu\). It is enough to show that \(k_\varepsilon\) is not \(\mu\)-a.e. constant. Now \(W\ge0\), and \(W(t)=0\) holds only for \(t\in\{0,t_0\}\). Since
\(\mu(\{0,t_0\})=0\), we have \(W>0\) \(\mu\)-a.e. If \(\beta>1\), then
\[
e^{-\beta W(t)/\varepsilon}<e^{-W(t)/\varepsilon}
\qquad\text{for }\mu\text{-a.e. }t,
\]
so
\[
\int_G e^{-\beta W(t)/\varepsilon}\,\mu(\dd t)
<
\int_G e^{-W(t)/\varepsilon}\,\mu(\dd t).
\]
If \(\beta<1\), the inequality is reversed. In both cases, the two values of
\(k_\varepsilon\) are different. We conclude that the first marginal of \(R_\varepsilon\) is not \(\mu\), and \(R_\varepsilon\notin\Pi(\mu,\mu)\).
\end{proof}

Recall the notation \(U_0,U_1\subset G\) and \(D(x,y)=y-x\). We define the four
partial partition functions
\[
I_0(\varepsilon):=\int_{U_0} e^{-W(t)/\varepsilon}\,\mu(\dd t),\qquad
I_1(\varepsilon):=\int_{U_1} e^{-W(t)/\varepsilon}\,\mu(\dd t),
\]
\[
J_0(\varepsilon):=\int_{U_0} e^{-\beta W(t)/\varepsilon}\,\mu(\dd t),\qquad
J_1(\varepsilon):=\int_{U_1} e^{-\beta W(t)/\varepsilon}\,\mu(\dd t),
\]
so
$I_0(\varepsilon),I_1(\varepsilon),J_0(\varepsilon),J_1(\varepsilon)\in(0,1/2]$,
because $\mu(U_0)=\mu(U_1)=1/2$, \(0<e^{-W/\varepsilon}\le 1\), and
\(0<e^{-\beta W/\varepsilon}\le 1\). We set
\begin{equation}\label{eq:cantor-AB-def}
A_\varepsilon:=\sqrt{I_0(\varepsilon)\,J_0(\varepsilon)},\qquad
B_\varepsilon:=\sqrt{I_1(\varepsilon)\,J_1(\varepsilon)}.
\end{equation}
\begin{equation}\label{eq:cantor-what-def}
w_\varepsilon:=\frac{A_\varepsilon}{A_\varepsilon+B_\varepsilon}\in(0,1).
\end{equation}
We also define the constants
\begin{equation}\label{eq:cantor-sinkhorn-constants}
\begin{aligned}
v_\varepsilon(0)&:=1,\qquad
v_\varepsilon(1):=\sqrt{\frac{I_0(\varepsilon)\,J_1(\varepsilon)}{J_0(\varepsilon)\,I_1(\varepsilon)}},\\
u_\varepsilon(0)&:=\frac{w_\varepsilon}{I_0(\varepsilon)},\qquad
u_\varepsilon(1):=\frac{1-w_\varepsilon}{J_1(\varepsilon)}.
\end{aligned}
\end{equation}
The block kernel mass matrix is
\begin{equation}\label{eq:cantor-kernel-matrix}
\mathsf K_\varepsilon:=
\begin{pmatrix}
I_0(\varepsilon) & I_1(\varepsilon)\\
J_1(\varepsilon) & J_0(\varepsilon)
\end{pmatrix},
\end{equation}
where the rows correspond to \(x_1\in\{0,1\}\) and the columns correspond to
\(y_1\in\{0,1\}\). Let us fix \(\varepsilon>0\), and let \(P_\varepsilon\) denote the unique minimiser of \(F_\varepsilon\) over \(\Pi(\mu,\mu)\).

\begin{lemma}[Diagonal \(H\)-invariance of the entropic minimiser]\label{lem:cantor-diagonal-H-invariance}
For a fixed \(\varepsilon>0\), and for \(s\in H\), set
\(\tau_s(x,y):=(x+s,y+s)\). It follows that \((\tau_s)_\#P_\varepsilon=P_\varepsilon\).
\end{lemma}

\begin{proof}
We fix \(s\in H\). Since \(s_1=0\), we have \(\gamma(x+s)=\gamma(x)\) and
\(W((y+s)-(x+s))=W(y-x)\), so \(C\circ\tau_s=C\). Moreover, \(\tau_s\) preserves \(\Lambda=\mu\otimes\mu\) and maps \(\Pi(\mu,\mu)\) into itself by Haar invariance. Since \(\tau_s\) is a bimeasurable bijection and \((\tau_s)_\#\Lambda=\Lambda\), relative entropy is invariant under \(\tau_s\), that is,
\(\mathrm{KL}((\tau_s)_\#P\|\Lambda)=\mathrm{KL}(P\|\Lambda)\) for all \(P\in\mathcal P(G^2)\). Therefore, for an arbitrary \(P\in\Pi(\mu,\mu)\),
\[
F_\varepsilon\bigl((\tau_s)_\#P\bigr)=F_\varepsilon(P).
\]
In particular, \((\tau_s)_\#P_\varepsilon\) is also a minimiser of
\(F_\varepsilon\). By uniqueness of the minimiser, we also get \((\tau_s)_\#P_\varepsilon=P_\varepsilon\).
\end{proof}

The formula follows by checking the marginal constraints and then
applying the Csisz{\'a}r projection criterion. This is the finite-dimensional
reduction we use below.

\subsection{Formula for the entropic minimiser}\label{sec:cantor-explicit}

We verify the \(2\times2\) scaling and then identify the new measure by
the Csisz{\'a}r projection criterion.

\begin{proposition}[Formula for the entropic minimiser]\label{prop:cantor-pi-eps-explicit}
For each \(\varepsilon>0\), the minimiser \(P_\varepsilon\) of
\(F_\varepsilon\) over \(\Pi(\mu,\mu)\) has the formula
\begin{equation}\label{eq:cantor-Peps-explicit}
P_\varepsilon(\dd x,\dd y)
=\mu(\dd x)\,\mu(\dd y)\,
u_\varepsilon(x_1)\,v_\varepsilon(y_1)\,
\exp\!\Big(-\frac{\gamma(x)W(y-x)}{\varepsilon}\Big),
\end{equation}
where $u_\varepsilon,v_\varepsilon$ are defined by
\eqref{eq:cantor-sinkhorn-constants}.
In other words,
\[
\frac{\dd P_\varepsilon}{\dd R_\varepsilon}(x,y)
=c_\varepsilon\,u_\varepsilon(x_1)\,v_\varepsilon(y_1).
\]
Moreover,
\begin{equation}\label{eq:cantor-w-identity}
P_\varepsilon\bigl(\{(x,y):(y-x)_1=0\}\bigr)=w_\varepsilon.
\end{equation}
\end{proposition}

\begin{proof}
Let us set
\[
\widetilde P_\varepsilon(\dd x,\dd y)
=\mu(\dd x)\,\mu(\dd y)\,
u_\varepsilon(x_1)\,v_\varepsilon(y_1)\,
\exp\!\Big(-\frac{\gamma(x)W(y-x)}{\varepsilon}\Big).
\]
For a density of the form
\[
u(x_1)v(y_1)\exp(-C(x,y)/\varepsilon),
\]
the marginal constraints reduce to diagonal scaling of
\(\mathsf K_\varepsilon\). Indeed, for fixed \(x\in G_i\), by Haar invariance and
the change of variables \(t=y-x\), we obtain
\[
\int_G v(y_1)\exp(-C(x,y)/\varepsilon)\,\mu(\dd y)
=
\sum_{j=0}^1(\mathsf K_\varepsilon)_{ij}v(j).
\]
Analogously, we fix \(y\in G_j\). If \(t:=y-x\), then \(x_1=i\) is equivalent to
\(t_1=j-i\pmod 2\), while \(\gamma(x)=1\) on \(G_0\) and \(\gamma(x)=\beta\)
on \(G_1\). It follows that
\[
\int_G u(x_1)v(y_1)\exp(-C(x,y)/\varepsilon)\,\mu(\dd x)
=
v(j)\sum_{i=0}^1 u(i)(\mathsf K_\varepsilon)_{ij}.
\]
We deduce that the measure has marginals \(\mu,\mu\) exactly when the associated
diagonal scaling of \(\mathsf K_\varepsilon\) has unit row and column sums.
By Lemma~\ref{lem:appendix-sinkhorn-scaling}, we observe that the choice
\eqref{eq:cantor-sinkhorn-constants} has this property, so \(\widetilde P_\varepsilon\in\Pi(\mu,\mu)\). Moreover,
\[
\frac{\dd \widetilde P_\varepsilon}{\dd\Lambda}(x,y)=f_\varepsilon(x)\,g_\varepsilon(y)\,
\exp\!\Big(-\frac{C(x,y)}{\varepsilon}\Big),
\qquad
f_\varepsilon(x):=u_\varepsilon(x_1),\ \ g_\varepsilon(y):=v_\varepsilon(y_1).
\]
Since
\[
\frac{\dd R_\varepsilon}{\dd\Lambda}(x,y)
=
c_\varepsilon^{-1}\exp\!\Bigl(-\frac{C(x,y)}{\varepsilon}\Bigr),
\]
we have
\[
\frac{\dd \widetilde P_\varepsilon}{\dd R_\varepsilon}(x,y)
=
a_\varepsilon(x)\,b_\varepsilon(y),
\qquad
a_\varepsilon(x):=c_\varepsilon\,u_\varepsilon(x_1),
\quad
b_\varepsilon(y):=v_\varepsilon(y_1).
\]
Notice that the functions \(a_\varepsilon\) and \(b_\varepsilon\) are bounded above and
below by positive constants. Lemma~\ref{lem:schrodinger-sufficient} implies
that \(\widetilde P_\varepsilon\) is the unique minimiser of \(F_\varepsilon\).
It follows that \(\widetilde P_\varepsilon=P_\varepsilon\), which proves
\eqref{eq:cantor-Peps-explicit}.

By the same change-of-variables method, with
Lemma~\ref{lem:appendix-sinkhorn-scaling}, we obtain
\[
P_\varepsilon(D\in U_0)
=
P_\varepsilon(\{(x,y):x_1=y_1\})
=
\frac12\Bigl(u_\varepsilon(0)v_\varepsilon(0)I_0(\varepsilon)
+u_\varepsilon(1)v_\varepsilon(1)J_0(\varepsilon)\Bigr)
=
w_\varepsilon,
\]
gives \eqref{eq:cantor-w-identity}.
\end{proof}

\paragraph{Csisz{\'a}r projection.}
In particular, \(P_\varepsilon\) is the \(I\)-projection of \(R_\varepsilon\) onto \(\Pi(\mu,\mu)\) in the sense of Csisz{\'a}r~\cite{Csiszar1975}.

\section{Oscillation and the interval cluster set}\label{sec:cantor-oscillation-main}
\subsection{Asymptotics of the block weight \texorpdfstring{$w_\varepsilon$}{w-epsilon}}\label{sec:cantor-oscillation}
Let
\[
f(x,y):=\onevec_{U_0}(y-x).
\]
Since \(U_0\subset G\) is clopen and \(D:G^2\to G\), \(D(x,y):=y-x\), is
continuous, we have \(f=\onevec_{U_0}\circ D\in C(G^2)\). By
Proposition~\ref{prop:cantor-pi-eps-explicit},
\begin{equation}\label{eq:cantor-test}
\int_{G^2} f\,\dd P_\varepsilon
=P_\varepsilon(D\in U_0)
=w_\varepsilon.
\end{equation}
Note that the weak convergence of \(P_\varepsilon\) would imply convergence of
\(w_\varepsilon\). It remains to prove that \(w_\varepsilon\) has no limit as
\(\varepsilon\downarrow0\). For $n\ge2$, the cylinder sets $A_n^0,A_n^1$ defined above satisfy the disjoint unions
\[
U_0=\{0\}\ \sqcup\ \bigsqcup_{n\ge2} A_n^0,
\qquad
U_1=\{t_0\}\ \sqcup\ \bigsqcup_{n\ge2} A_n^1,
\]
and $\mu(\{0\})=\mu(\{t_0\})=0$. Moreover, $\mu(A_n^0)=\mu(A_n^1)=2^{-n}$, \(W\equiv\alpha_n\) on \(A_n^0\), and \(W\equiv\alpha_nL_n\) on \(A_n^1\). We obtain the exact series representations
\begin{equation}\label{eq:cantor-series-new}
I_0(\varepsilon)=\sum_{n\ge2}2^{-n}\exp\!\Big(-\frac{\alpha_n}{\varepsilon}\Big),
\qquad
I_1(\varepsilon)=\sum_{n\ge2}2^{-n}\exp\!\Big(-\frac{\alpha_nL_n}{\varepsilon}\Big),
\end{equation}
\begin{equation}\label{eq:cantor-series-new-beta}
J_0(\varepsilon)=\sum_{n\ge2}2^{-n}\exp\!\Big(-\frac{\beta\alpha_n}{\varepsilon}\Big),
\qquad
J_1(\varepsilon)=\sum_{n\ge2}2^{-n}\exp\!\Big(-\frac{\beta\alpha_nL_n}{\varepsilon}\Big).
\end{equation}

\begin{lemma}[Asymptotics along $\varepsilon_n=\alpha_n$]\label{lem:cantor-asymptotics}
Let $\varepsilon_n:=\alpha_n=\exp(-2^n)$ for $n\ge2$. As \(n\to\infty\),
\[
I_0(\varepsilon_n)=2^{-n}\bigl(1+e^{-1}\bigr)+o(2^{-n}),
\qquad
I_1(\varepsilon_n)=2^{-n}\bigl(1+e^{-L_n}\bigr)+o(2^{-n}),
\]
\[
J_0(\varepsilon_n)=2^{-n}\bigl(1+e^{-\beta}\bigr)+o(2^{-n}),
\qquad
J_1(\varepsilon_n)=2^{-n}\bigl(1+e^{-\beta L_n}\bigr)+o(2^{-n}).
\]
Observe that the $o(2^{-n})$ terms are uniform for $L_n\in[1-a,1+a]$. With
$w_\varepsilon$, as defined in \eqref{eq:cantor-what-def},
\[
w_{\varepsilon_n}
=\frac{\sqrt{(1+e^{-1})(1+e^{-\beta})}}{\sqrt{(1+e^{-1})(1+e^{-\beta})}+\sqrt{(1+e^{-L_n})(1+e^{-\beta L_n})}}+o(1).
\]
\end{lemma}

\begin{proof}
We fix $n\ge2$ and set $\varepsilon_n:=\alpha_n$. We prove one uniform estimate
and then apply it to the four series.

\begin{claim}
Let $(b_m)_{m\ge2}$ be a sequence with $0<b_*\le b_m\le b^*<\infty$. We define
\[
S_b(\varepsilon):=\sum_{m\ge2}2^{-m}\exp\!\Big(-\frac{b_m\alpha_m}{\varepsilon}\Big).
\]
As $n\to\infty$,
\[
S_b(\varepsilon_n)=2^{-n}\bigl(1+e^{-b_n}\bigr)+o(2^{-n}),
\]
where the $o(2^{-n})$ is uniform over all sequences $(b_m)$ with the fixed bounds $[b_*,b^*]$.
\end{claim}

\begin{proof}
We split
\[
S_b(\varepsilon_n)=\sum_{m<n}2^{-m}e^{-b_m\alpha_m/\varepsilon_n}+2^{-n}e^{-b_n}
+\sum_{m>n}2^{-m}e^{-b_m\alpha_m/\varepsilon_n}.
\]
Since $\alpha_m/\varepsilon_n=\exp(-(2^m-2^n))$, for $m>n$, we have
$0\le \alpha_m/\varepsilon_n\le e^{-2^n}$. We use $|e^{-u}-1|\le u$ for
$u\ge0$ and $b_m\le b^*$,
\[
\Big|e^{-b_m\alpha_m/\varepsilon_n}-1\Big|
\le b^*\,\frac{\alpha_m}{\varepsilon_n}
\le b^*e^{-2^n},
\]
so
\[
\sum_{m>n}2^{-m}e^{-b_m\alpha_m/\varepsilon_n}
=\sum_{m>n}2^{-m}+O\!\Big(e^{-2^n}\sum_{m>n}2^{-m}\Big)
=2^{-n}+O(2^{-n}e^{-2^n}).
\]
For $m<n$, we have $2^n-2^m\ge 2^{n-1}$, so
$\alpha_m/\varepsilon_n=\exp(2^n-2^m)\ge \exp(2^{n-1})$. This implies
\[
0\le \sum_{m<n}2^{-m}e^{-b_m\alpha_m/\varepsilon_n}
\le \sum_{m<n}2^{-m}e^{-b_*\exp(2^{n-1})}
\le e^{-b_*\exp(2^{n-1})}
=o(2^{-n}).
\]
It suffices to combine the two estimates to prove the claim.
\end{proof}

We apply the claim with:
\[
(I_0):\ b_m\equiv 1;\qquad
(I_1):\ b_m=L_m;\qquad
(J_0):\ b_m\equiv \beta;\qquad
(J_1):\ b_m=\beta L_m.
\]
Since $L_m\in[1-a,1+a]$, the four choices satisfy uniform bounds of the
required form. We get
\[
I_0(\varepsilon_n)=2^{-n}(1+e^{-1})+o(2^{-n}),\qquad
I_1(\varepsilon_n)=2^{-n}(1+e^{-L_n})+o(2^{-n}),
\]
\[
J_0(\varepsilon_n)=2^{-n}(1+e^{-\beta})+o(2^{-n}),\qquad
J_1(\varepsilon_n)=2^{-n}(1+e^{-\beta L_n})+o(2^{-n}),
\]
with uniform $o(2^{-n})$ over $L_n\in[1-a,1+a]$. We now insert these estimates into \eqref{eq:cantor-what-def}. Since $A_{\varepsilon_n}=\sqrt{I_0(\varepsilon_n)J_0(\varepsilon_n)}$ and $B_{\varepsilon_n}=\sqrt{I_1(\varepsilon_n)J_1(\varepsilon_n)}$, we divide numerator and denominator of
$w_{\varepsilon_n}=A_{\varepsilon_n}/(A_{\varepsilon_n}+B_{\varepsilon_n})$ by
$2^{-n}$ and obtain the asymptotic formula in question.
\end{proof}

Observe that $L_n=1+a$ for even $n$ and $L_n=1-a$ for odd $n$, so Lemma~\ref{lem:cantor-asymptotics} yields
\begin{equation}\label{eq:cantor-wpm}
\lim_{k\to\infty} w_{\varepsilon_{2k}}=w^+,\qquad
\lim_{k\to\infty} w_{\varepsilon_{2k+1}}=w^-,
\end{equation}
and Corollary~\ref{cor:cantor-wpm-closed} gives the closed forms for \(w^\pm\).

\begin{corollary}[Closed forms for the alternating limits]\label{cor:cantor-wpm-closed}
The limits in \eqref{eq:cantor-wpm} are
\begin{equation}\label{eq:cantor-wpm-closed}
w^\pm
=\frac{\sqrt{(1+e^{-1})(1+e^{-\beta})}}{
\sqrt{(1+e^{-1})(1+e^{-\beta})}
+\sqrt{(1+e^{-(1\pm a)})(1+e^{-\beta(1\pm a)})}}.
\end{equation}
\end{corollary}

\begin{proof}
It suffices to substitute \(L=1\pm a\) in the asymptotic expression for \(w_{\varepsilon_n}\) from Lemma~\ref{lem:cantor-asymptotics}.
\end{proof}

\begin{proposition}[Oscillation of the scalar branch weight]\label{prop:cantor-nonconvergence}
The curve \((w_\varepsilon)_{\varepsilon>0}\) has at least two distinct cluster
points. More precisely, \(w^+>w^-\).
\end{proposition}

\begin{proof}
Observe that the function
\[
L\longmapsto \sqrt{(1+e^{-L})(1+e^{-\beta L})}
\]
is strictly decreasing on \((0,\infty)\), due to the fact that both factors are positive and
strictly decreasing. Since \(1+a>1-a\), by Corollary~\ref{cor:cantor-wpm-closed},
we have \(w^+>w^-\). By \eqref{eq:cantor-wpm}, these are distinct cluster points
of the scalar curve. In other terms, \eqref{eq:cantor-test} yields
\[
\int f\,\dd P_{\varepsilon_{2k}}\to w^+,
\qquad
\int f\,\dd P_{\varepsilon_{2k+1}}\to w^-.
\]
We conclude that a continuous test function separates the respective subsequential
limits of the plans.
\end{proof}

\subsection{Identification of all cluster points}\label{sec:cantor-cluster}
Let us set
\[
\mathcal W:=\Clust(w_\varepsilon)\subset[0,1],
\qquad
P_w:=wP_0+(1-w)P_{t_0}.
\]
On this segment, we can recover the parameter by the graph mass
\[
w=P_w\{(x,y):y=x\},
\]
which implies that the map \(w\mapsto P_w\) is injective. Put differently, for the continuous test function \(f=\onevec_{U_0}\circ D\) we introduced in
\eqref{eq:cantor-test}, we obtain \(\int f\,\dd P_w=w\). We identify the full cluster set by observing the scalar weight. Proposition~\ref{prop:cluster-optimal-Polish} and
Lemmas~\ref{lem:cantor-diagonal-H-invariance}--\ref{lem:cantor-h-classification}
imply that every weak cluster point of \((P_\varepsilon)\) has the form
\(P_w\) for some \(w\in[0,1]\). It remains to determine which weights appear
along the sequences \(\varepsilon\downarrow0\). We apply here the classification Lemma~\ref{lem:cantor-h-classification} (proved in Appendix~\ref{app:H-invariant-face}).

\begin{proposition}[Subsequential limits]\label{prop:cantor-all-limits}
Let \(\varepsilon_k\downarrow0\), and assume that \(w_{\varepsilon_k}\to w\in[0,1]\). It follows that \(P_{\varepsilon_k}\Rightarrow P_w\). In particular,
\[
\Clust(P_\varepsilon)=\{P_w:w\in\mathcal W\}.
\]
\end{proposition}

\begin{proof}
We fix a subsequence of \((P_{\varepsilon_k})\). By weak compactness of
\(\Pi(\mu,\mu)\), we may pass to a further subsequence, without relabelling, such
that
\[
P_{\varepsilon_k}\Rightarrow P
\qquad\text{for some }P\in\Pi(\mu,\mu).
\]
For every \(s\in H\), Lemma~\ref{lem:cantor-diagonal-H-invariance} gives
\[
(\tau_s)_\#P_{\varepsilon_k}=P_{\varepsilon_k}.
\]
Passing to the limit gives \((\tau_s)_\#P=P\), so \(P\) is diagonally
\(H\)-invariant.

By Proposition~\ref{prop:cluster-optimal-Polish}, \(P\) is optimal, so
Proposition~\ref{prop:cantor-opt-face} gives \(P(Z)=1\). We apply
Lemma~\ref{lem:cantor-h-classification}, and obtain
\[
P=P_u
\qquad\text{for some }u\in[0,1].
\]
Let \(f\in C(G^2)\) be the test function from \eqref{eq:cantor-test}. By Proposition~\ref{prop:cantor-pi-eps-explicit},
\[
\int_{G^2} f\,\dd P_{\varepsilon}=w_\varepsilon.
\]
Therefore,
\[
u=\int_{G^2} f\,\dd P_u
=\lim_{k\to\infty}\int_{G^2} f\,\dd P_{\varepsilon_k}
=\lim_{k\to\infty} w_{\varepsilon_k}
=w.
\]
This implies that every subsequence of \((P_{\varepsilon_k})\) has another subsequence convering to \(P_w\). Since \(\Pi(\mu,\mu)\) is compact metrisable, it follows that
\[
P_{\varepsilon_k}\Rightarrow P_w.
\]

To determine the full cluster set, let \(P\in\Clust(P_\varepsilon)\), and choose
\(\varepsilon_k\downarrow0\) such that \(P_{\varepsilon_k}\Rightarrow P\).
After passing to a subsequence, we may assume, in addition, that
\(w_{\varepsilon_k}\to \bar w\in\mathcal W\). The first part gives
\(P=P_{\bar w}\), so
\[
\Clust(P_\varepsilon)\subset\{P_w:w\in\mathcal W\}.
\]
Conversely, let \(w\in\mathcal W\), and choose \(\varepsilon_k\downarrow0\) such
that \(w_{\varepsilon_k}\to w\). The first part gives
\(P_{\varepsilon_k}\Rightarrow P_w\), so \(P_w\in\Clust(P_\varepsilon)\).
\[
\Clust(P_\varepsilon)=\{P_w:w\in\mathcal W\}.
\]
\end{proof}

\subsection{Interval structure and proof of the Cantor theorem}\label{sec:cantor-proof-main}

\begin{theorem}[Characterisation of the cluster set in the Cantor model]\label{thm:cantor-cluster}
Let us fix \(a\in(0,1)\) and \(\beta\in(0,\infty)\setminus\{1\}\). In the model from
Definition~\ref{def:cantor-explicit-model}, let \(P_\varepsilon\) denote the
unique entropic minimiser, and let \(w_\varepsilon\) be the scalar from
\eqref{eq:cantor-what-def} (or from \eqref{eq:cantor-w-identity}). It follows that
\[
\Clust(P_\varepsilon)=\{P_w:w\in\mathcal W\}.
\]
Moreover, \(\mathcal W\) is a non-degenerate compact interval. Along
\(\varepsilon_n=e^{-2^n}\),
\[
w_{\varepsilon_{2n}}\to w^+,\qquad
w_{\varepsilon_{2n+1}}\to w^-,
\qquad
w^+>w^-.
\]
In particular, \((P_\varepsilon)\) does not converge weakly as \(\varepsilon\downarrow0\).
\end{theorem}

\begin{lemma}[Interval structure of $\mathcal W$]\label{lem:cantor-W-interval}
The set $\mathcal W=\Clust(w_\varepsilon)$ is a nonempty compact interval in
$[0,1]$. Moreover, \(w^-\) and \(w^+\) belong to \(\mathcal W\), thus
$[w^-,w^+]\subset\mathcal W$.
\end{lemma}

\begin{proof}
By Theorem~\ref{thm:selection-connected}, we know that the cluster set
\(\Clust(P_\varepsilon)\) is a nonempty compact connected subset of
\(\Pi(\mu,\mu)\). By Proposition~\ref{prop:cantor-all-limits},
\[
\Clust(P_\varepsilon)=\{P_w:w\in\mathcal W\}.
\]
Let \(f\in C(G^2)\) be the test function from \eqref{eq:cantor-test}. This means that the
map
\[
\Clust(P_\varepsilon)\ni P \longmapsto \int_{G^2} f\,\dd P
\]
is continuous, and for every \(w\in\mathcal W\),
\[
\int_{G^2} f\,\dd P_w=w.
\]
We notice that \(\mathcal W\) is the continuous image of the connected compact set
\(\Clust(P_\varepsilon)\), so \(\mathcal W\) is a nonempty compact connected
subset of \([0,1]\), and also a compact interval. By \eqref{eq:cantor-wpm}, both \(w^-\) and \(w^+\) belong to \(\mathcal W\). Since \(w^-<w^+\), we get
\[
[w^-,w^+]\subset\mathcal W.
\]
\end{proof}

\begin{remark}
We could also establish the interval property of \(\mathcal W\) directly
from continuity of \(\varepsilon\mapsto w_\varepsilon\) and apply the same nested
continuum argument to the scalar curve \((0,\infty)\ni\varepsilon\mapsto w_\varepsilon\in[0,1]\). We use here Theorem~\ref{thm:selection-connected} because it gives the statement at the level of transport plans.
\end{remark}

\begin{proof}[Proof of Theorem~\ref{thm:cantor-cluster}]
Proposition~\ref{prop:cantor-all-limits} gives
\[
\Clust(P_\varepsilon)=\{P_w:w\in\mathcal W\},
\]
and Lemma~\ref{lem:cantor-W-interval} shows that \(\mathcal W\) is a non-degenerate
compact interval in \([0,1]\). Proposition~\ref{prop:cantor-nonconvergence}, together with
\eqref{eq:cantor-wpm} and Corollary~\ref{cor:cantor-wpm-closed}, yields
\[
w_{\varepsilon_{2n}}\to w^+,\qquad
w_{\varepsilon_{2n+1}}\to w^-,
\qquad
w^+>w^-.
\]
Recall that \(P_{w^+}\ne P_{w^-}\), so the nonconvergence statement follows immediately.
\end{proof}

\begin{proof}[Proof of Theorem~\ref{thm:intro-counterexample}]
We take \(X=G=\{0,1\}^{\mathbb N}\) with the dyadic ultrametric, and let
\[
\mu=\Bigl(\tfrac12\delta_0+\tfrac12\delta_1\Bigr)^{\otimes\mathbb N}.
\]
This implies that \(X\) is compact and \(\mu\) is atomless. By Lemmas~\ref{lem:cantor-continuity} and~\ref{lem:cantor-lipschitz}, the cost \(C\) is bounded and Lipschitz on \(G^2\). Proposition~\ref{prop:cantor-opt-face} shows that the unregularised problem has multiple minimisers and that every optimal plan is singular with respect to \(\mu\otimes\mu\). By Theorem~\ref{thm:cantor-cluster}, we obtain the interval description of \(\Clust(P_\varepsilon)\). In particular, \(P_\varepsilon\) does not converge weakly as \(\varepsilon\downarrow0\).
\end{proof}

\appendix

\section{Variational criteria for cluster membership and full convergence}
\label{sec:selection-criteria}
We prove two first-order variational criteria in a general Polish optimal-transport problem: one for cluster membership and one for full convergence. We keep the same
assumptions and notation of Section~\ref{sec:general-theory}. We first establish the local criterion, and then the exterior one. We use the bounded-Lipschitz metric \(d_{\mathrm{BL}}\), which metrises weak convergence on the coupling set. Let \(S:=X\times Y\), and fix a complete metric \(\varrho\) on \(S\) that generates its topology and satisfies
\[
0\le \varrho\le 1.
\]
Here,
\[
\mathrm{Lip}_{\varrho}(f):=\sup_{\substack{z,z'\in S\\ z\neq z'}}\frac{|f(z)-f(z')|}{\varrho(z,z')}.
\]
We also define the bounded-Lipschitz metric on \(\mathcal P(S)\) by
\[
d_{\mathrm{BL}}(P,Q)
:=
\sup\Bigl\{
\Big|\int_S f\,\dd(P-Q)\Big|:
f\in C_b(S),\ \|f\|_\infty\le 1,\ \mathrm{Lip}_{\varrho}(f)\le 1
\Bigr\}.
\]
We use the fact that \(d_{\mathrm{BL}}\) metrises weak convergence on
\(\mathcal P(S)\), and write
\[
d_{\mathrm{TV}}(P,Q):=\sup_A|P(A)-Q(A)|
\]
for total variation distance. For every \(P,Q\in\mathcal P(S)\),
\begin{equation}\label{eq:universal-bl-tv}
d_{\mathrm{BL}}(P,Q)\le 2\,d_{\mathrm{TV}}(P,Q).
\end{equation}
Indeed, for every admissible \(f\) in the definition of \(d_{\mathrm{BL}}\),
\[
\left|\int_S f\,\dd(P-Q)\right|
\le
\sup_{\|g\|_\infty\le1}\left|\int_S g\,\dd(P-Q)\right|
=
2\,d_{\mathrm{TV}}(P,Q).
\]
It suffices to take the supremum over \(f\) to obtain \eqref{eq:universal-bl-tv}.

Let
\[
\mathcal V(\varepsilon)
:=
\inf_{\pi\in\Pi(\mu,\nu)}
\bigl(\Delta(\pi)+\varepsilon\Ent(\pi)\bigr)
=
F_\varepsilon(\pi_\varepsilon)-C_0.
\]
For \(\bar\pi\in\Pi(\mu,\nu)\), \(r>0\), and \(\varepsilon>0\), let
\[
\mathcal V_{\bar\pi,r}(\varepsilon)
:=
\inf\Bigl\{
\Delta(\pi)+\varepsilon \Ent(\pi):
\pi\in\Pi(\mu,\nu),\ d_{\mathrm{BL}}(\pi,\bar\pi)\le r
\Bigr\}\in[0,\infty].
\]
By definition, we obtain
\[
\mathcal V(\varepsilon)\le \mathcal V_{\bar\pi,r}(\varepsilon)
\qquad\text{for all }\bar\pi\in\Pi(\mu,\nu),\ r>0,\ \varepsilon>0.
\]
For \(\bar\pi\in\Pi(\mu,\nu)\), \(r>0\), and \(\varepsilon>0\), let
\[
\mathcal V^{\mathrm{out}}_{\bar\pi,r}(\varepsilon)
:=
\inf\Bigl\{
\Delta(\pi)+\varepsilon \Ent(\pi):
\pi\in\Pi(\mu,\nu),\ d_{\mathrm{BL}}(\pi,\bar\pi)\ge r
\Bigr\},
\]
with the convention \(\inf\varnothing:=+\infty\).

\begin{theorem}[Local and exterior selection criteria]\label{thm:selection-criteria}
We assume the hypotheses from Section~\ref{sec:entropic-prelim}, and let
\(\bar\pi\in\Pi(\mu,\nu)\).
\begin{enumerate}[label=(\roman*),leftmargin=*]
\item It is true that \(\bar\pi\in\Omega(\mu,\nu\mid C)\) if and only if, for every \(r>0\),
\[
\liminf_{\varepsilon\downarrow0}
\frac{\mathcal V_{\bar\pi,r}(\varepsilon)-\mathcal V(\varepsilon)}{\varepsilon}=0.
\]
\item \(\pi_\varepsilon\Rightarrow\bar\pi\) as \(\varepsilon\downarrow0\) holds if and only if, for every \(r>0\),
\[
\liminf_{\varepsilon\downarrow0}
\frac{\mathcal V^{\mathrm{out}}_{\bar\pi,r}(\varepsilon)-\mathcal V(\varepsilon)}{\varepsilon}>0.
\]
\end{enumerate}
\end{theorem}

\begin{remark}
Since \(d_{\mathrm{BL}}\) metrises weak convergence on \(\Pi(\mu,\nu)\), the
metric-ball formulation is equivalent to the usual one in terms of weak
neighbourhoods.
\end{remark}

\begin{lemma}[Existence of local minimisers]\label{lem:local-minimisers-universal}
For every \(\bar\pi\in\Pi(\mu,\nu)\), every \(r>0\), and every
\(\varepsilon>0\), the value \(\mathcal V_{\bar\pi,r}(\varepsilon)\) is
attained.
\end{lemma}

\begin{proof}
Since \(d_{\mathrm{BL}}\) metrises the weak topology on \(\mathcal P(S)\), the
closed ball
\[
B_r(\bar\pi):=\{\pi\in\Pi(\mu,\nu):d_{\mathrm{BL}}(\pi,\bar\pi)\le r\}
\]
is weakly closed in \(\Pi(\mu,\nu)\).
By Proposition~\ref{prop:standard-eot-facts}, we already know that \(B_r(\bar\pi)\) is weakly compact. The weak lower semicontinuity of \(J\) and \(\Ent\), again from
Proposition~\ref{prop:standard-eot-facts}, implies that the map
\[
\pi\longmapsto \Delta(\pi)+\varepsilon \Ent(\pi)
\]
is weakly lower semicontinuous on \(B_r(\bar\pi)\). This implies that the infimum is
attained.
\end{proof}

\begin{lemma}[Strong stability of the entropic minimiser]\label{lem:strong-stability-universal}
For every \(\varepsilon>0\) and \(\pi\in\Pi(\mu,\nu)\), it is true that
\begin{equation}\label{eq:universal-stability-tv}
F_\varepsilon(\pi)-F_\varepsilon(\pi_\varepsilon)
\ge
\varepsilon\, d_{\mathrm{TV}}(\pi,\pi_\varepsilon)^2.
\end{equation}
This implies
\begin{equation}\label{eq:universal-stability-bl}
F_\varepsilon(\pi)-F_\varepsilon(\pi_\varepsilon)
\ge
\frac{\varepsilon}{4}\, d_{\mathrm{BL}}(\pi,\pi_\varepsilon)^2.
\end{equation}
\end{lemma}

\begin{proof}
If \(F_\varepsilon(\pi)=+\infty\), then \eqref{eq:universal-stability-tv} is
trivial. Let us assume \(F_\varepsilon(\pi)<\infty\). This means that \(J(\pi)<\infty\) and
\(\Ent(\pi)<\infty\), so \(\pi\ll\Lambda\). By Proposition~\ref{prop:standard-eot-facts}, also \(\pi_\varepsilon\ll\Lambda\). We set
\[
M:=\frac{\pi+\pi_\varepsilon}{2}\in\Pi(\mu,\nu).
\]
Given that \(J(\pi)\) and \(J(\pi_\varepsilon)\) are finite,
\[
J(M)=\frac12J(\pi)+\frac12J(\pi_\varepsilon).
\]
Let
\[
p:=\frac{\dd\pi}{\dd\Lambda},\qquad
q:=\frac{\dd\pi_\varepsilon}{\dd\Lambda},\qquad
m:=\frac{\dd M}{\dd\Lambda}=\frac{p+q}{2}.
\]
We use the convention \(0\log0=0\). On \(\{p+q=0\}\), we set all terms below equal to \(0\). On \(\{p+q>0\}\), the scalar identity we consider next is exact. Moreover, \(p/m\le2\) and \(q/m\le2\), so the two relative-entropy correction terms are bounded above by integrable multiples of \(p\) and \(q\). This implies that we may integrate the identity without ambiguity w.r.t. extended values. The pointwise identity
\[
m\log m
=
\frac12\,p\log p+\frac12\,q\log q
-\frac12\,p\log\frac{p}{m}
-\frac12\,q\log\frac{q}{m}
\]
holds. Once we have integrated against \(\Lambda\), we obtain
\[
\Ent(M)
=
\frac12\Ent(\pi)+\frac12\Ent(\pi_\varepsilon)
-\frac12\mathrm{KL}(\pi\|M)
-\frac12\mathrm{KL}(\pi_\varepsilon\|M).
\]
Therefore,
\[
F_\varepsilon(M)
=
\frac12F_\varepsilon(\pi)+\frac12F_\varepsilon(\pi_\varepsilon)
-\frac{\varepsilon}{2}\mathrm{KL}(\pi\|M)
-\frac{\varepsilon}{2}\mathrm{KL}(\pi_\varepsilon\|M).
\]
Since \(\pi_\varepsilon\) minimises \(F_\varepsilon\),
\[
F_\varepsilon(\pi_\varepsilon)\le F_\varepsilon(M).
\]
It follows that
\[
F_\varepsilon(\pi)-F_\varepsilon(\pi_\varepsilon)
\ge
\varepsilon\,\mathrm{KL}(\pi\|M)+\varepsilon\,\mathrm{KL}(\pi_\varepsilon\|M).
\]
By Pinsker's inequality, we obtain
\[
\mathrm{KL}(\pi\|M)\ge 2\,d_{\mathrm{TV}}(\pi,M)^2,
\qquad
\mathrm{KL}(\pi_\varepsilon\|M)\ge 2\,d_{\mathrm{TV}}(\pi_\varepsilon,M)^2.
\]
Since \(M=(\pi+\pi_\varepsilon)/2\), we have
\[
d_{\mathrm{TV}}(\pi,M)=d_{\mathrm{TV}}(\pi_\varepsilon,M)=\frac12\,d_{\mathrm{TV}}(\pi,\pi_\varepsilon).
\]
Therefore,
\[
F_\varepsilon(\pi)-F_\varepsilon(\pi_\varepsilon)
\ge
\varepsilon\, d_{\mathrm{TV}}(\pi,\pi_\varepsilon)^2.
\]
This proves \eqref{eq:universal-stability-tv}. Finally, by \eqref{eq:universal-bl-tv},
\[
d_{\mathrm{TV}}(\pi,\pi_\varepsilon)\ge \frac12 d_{\mathrm{BL}}(\pi,\pi_\varepsilon),
\]
and \eqref{eq:universal-stability-bl} follows immediately.
\end{proof}

For \(\bar\pi\in\Pi(\mu,\nu)\) and \(r>0\), we define
\[
\mathfrak G_r(\bar\pi)
:=
\liminf_{\varepsilon\downarrow0}
\frac{\mathcal V_{\bar\pi,r}(\varepsilon)-\mathcal V(\varepsilon)}{\varepsilon}
\in[0,\infty],
\]
and set
\[
\mathfrak G(\bar\pi):=\sup_{r>0}\mathfrak G_r(\bar\pi)\in[0,\infty].
\]

\begin{lemma}[Monotonicity in the radius]\label{lem:radius-monotone-universal}
For every \(\bar\pi\in\Pi(\mu,\nu)\), the function
\[
(0,\infty)\ni r\longmapsto \mathfrak G_r(\bar\pi)
\]
is nonincreasing. This implies
\[
\mathfrak G(\bar\pi)=\sup_{m\ge1}\mathfrak G_{1/m}(\bar\pi).
\]
\end{lemma}

\begin{proof}
If \(0<r_1\le r_2\), then
\[
\{\pi:d_{\mathrm{BL}}(\pi,\bar\pi)\le r_1\}
\subset
\{\pi:d_{\mathrm{BL}}(\pi,\bar\pi)\le r_2\},
\]
so
\[
\mathcal V_{\bar\pi,r_1}(\varepsilon)\ge \mathcal V_{\bar\pi,r_2}(\varepsilon)
\qquad\text{for every }\varepsilon>0.
\]
It suffices to subtract \(\mathcal V(\varepsilon)\), divide by \(\varepsilon\), and
take \(\liminf_{\varepsilon\downarrow0}\), which yields
\[
\mathfrak G_{r_1}(\bar\pi)\ge \mathfrak G_{r_2}(\bar\pi).
\]
The formula for \(\mathfrak G(\bar\pi)\) follows immediately.
\end{proof}

\begin{proposition}[Local criterion for cluster points]\label{prop:canonical-local-gap}
Using the assumptions of this section,
\[
\Omega(\mu,\nu\mid C)
=
\{\bar\pi\in\Pi(\mu,\nu):\mathfrak G(\bar\pi)=0\}.
\]
In other terms,
\[
\bar\pi\in\Omega(\mu,\nu\mid C)
\quad\Longleftrightarrow\quad
\forall r>0,\qquad
\liminf_{\varepsilon\downarrow0}
\frac{\mathcal V_{\bar\pi,r}(\varepsilon)-\mathcal V(\varepsilon)}{\varepsilon}=0.
\]
\end{proposition}

\begin{proof}
Let \(\bar\pi\in\Omega(\mu,\nu\mid C)\), and choose \(\varepsilon_n\downarrow0\) such that \(\pi_{\varepsilon_n}\Rightarrow\bar\pi\). Since \(d_{\mathrm{BL}}\) metrises
weak convergence, \(d_{\mathrm{BL}}(\pi_{\varepsilon_n},\bar\pi)\to0\). We fix
\(r>0\). For all sufficiently large \(n\), \(\pi_{\varepsilon_n}\) is
admissible in the local problem that defines \(\mathcal V_{\bar\pi,r}(\varepsilon_n)\). This implies
\[
\mathcal V_{\bar\pi,r}(\varepsilon_n)\le
\Delta(\pi_{\varepsilon_n})+\varepsilon_n \Ent(\pi_{\varepsilon_n})
=
\mathcal V(\varepsilon_n).
\]
Since
\(\mathcal V(\varepsilon_n)\le \mathcal V_{\bar\pi,r}(\varepsilon_n)\)
always holds, equality also holds for all large \(n\). We deduce that
\(\mathfrak G_r(\bar\pi)=0\). Recall that \(r>0\) was arbitrary, so \(\mathfrak G(\bar\pi)=0\).

For the converse, we assume that \(\mathfrak G(\bar\pi)=0\). By Lemma~\ref{lem:radius-monotone-universal},
\[
\mathfrak G_{1/n}(\bar\pi)=0
\qquad\text{for every }n\ge1.
\]
We construct a strictly decreasing sequence \((\varepsilon_n)\) such that
\[
0<\varepsilon_n<\min\Bigl\{\frac1n,\frac{\varepsilon_{n-1}}{2}\Bigr\}
\qquad (n\ge2),
\]
and
\begin{equation}\label{eq:epsn-choice-universal}
\mathcal V_{\bar\pi,1/n}(\varepsilon_n)-\mathcal V(\varepsilon_n)\le \frac{\varepsilon_n}{n^2}.
\end{equation}
For \(n=1\), we choose any \(\varepsilon_1\in(0,1)\) that satisfies
\eqref{eq:epsn-choice-universal}. Suppose that \(\varepsilon_{n-1}\) has been
chosen. Since \(\mathfrak G_{1/n}(\bar\pi)=0\), there exist arbitrarily small
\(\varepsilon>0\) such that
\[
\mathcal V_{\bar\pi,1/n}(\varepsilon)-\mathcal V(\varepsilon)\le \frac{\varepsilon}{n^2}.
\]
Let us choose one with
\[
0<\varepsilon<\min\Bigl\{\frac1n,\frac{\varepsilon_{n-1}}{2}\Bigr\},
\]
and call it \(\varepsilon_n\). This implies \(\varepsilon_n\downarrow0\), and
\eqref{eq:epsn-choice-universal} holds for every \(n\).

By Lemma~\ref{lem:local-minimisers-universal}, for each \(n\) there exists
\(\eta_n\in\Pi(\mu,\nu)\) such that
\[
d_{\mathrm{BL}}(\eta_n,\bar\pi)\le \frac1n
\]
and
\[
\Delta(\eta_n)+\varepsilon_n \Ent(\eta_n)=\mathcal V_{\bar\pi,1/n}(\varepsilon_n).
\]
Since
\[
\Delta(\pi_{\varepsilon_n})+\varepsilon_n \Ent(\pi_{\varepsilon_n})=\mathcal V(\varepsilon_n),
\]
estimate \eqref{eq:epsn-choice-universal} yields
\[
F_{\varepsilon_n}(\eta_n)-F_{\varepsilon_n}(\pi_{\varepsilon_n})
=
\mathcal V_{\bar\pi,1/n}(\varepsilon_n)-\mathcal V(\varepsilon_n)
\le \frac{\varepsilon_n}{n^2}.
\]
We apply Lemma~\ref{lem:strong-stability-universal}, and obtain
\[
\frac{\varepsilon_n}{4}\,d_{\mathrm{BL}}(\eta_n,\pi_{\varepsilon_n})^2
\le
F_{\varepsilon_n}(\eta_n)-F_{\varepsilon_n}(\pi_{\varepsilon_n})
\le \frac{\varepsilon_n}{n^2},
\]
so \(d_{\mathrm{BL}}(\eta_n,\pi_{\varepsilon_n})\le 2/n\). Therefore,
\[
d_{\mathrm{BL}}(\pi_{\varepsilon_n},\bar\pi)
\le
d_{\mathrm{BL}}(\pi_{\varepsilon_n},\eta_n)+d_{\mathrm{BL}}(\eta_n,\bar\pi)
\le
\frac2n+\frac1n=\frac3n.
\]
The above implies \(d_{\mathrm{BL}}(\pi_{\varepsilon_n},\bar\pi)\to0\), and
\(\pi_{\varepsilon_n}\Rightarrow\bar\pi\). Since \(\varepsilon_n\downarrow0\),
we conclude that \(\bar\pi\in\Omega(\mu,\nu\mid C)\).
\end{proof}

\subsection{Variational criterion for full convergence}\label{subsec:omega-exterior-gap}

We now prove the exterior-gap criterion.

\begin{proposition}[Exterior-gap criterion for full convergence]
\label{prop:canonical-exterior-gap}
Under the assumptions of this section, for \(\bar\pi\in\Pi(\mu,\nu)\), the
following statements are equivalent:
\begin{enumerate}[label=(\roman*),leftmargin=*]
\item \(\pi_\varepsilon\Rightarrow \bar\pi\) as \(\varepsilon\downarrow0\);
\item for every \(r>0\),
\[
\liminf_{\varepsilon\downarrow0}
\frac{\mathcal V^{\mathrm{out}}_{\bar\pi,r}(\varepsilon)-\mathcal V(\varepsilon)}{\varepsilon}>0.
\]
\end{enumerate}
\end{proposition}

\begin{proof}[Proof of Proposition~\ref{prop:canonical-exterior-gap}]
We shall first assume that \(\pi_\varepsilon\Rightarrow \bar\pi\) as \(\varepsilon\downarrow0\), and fix \(r>0\). For all sufficiently small \(\varepsilon\),
\[
d_{\mathrm{BL}}(\pi_\varepsilon,\bar\pi)<\frac r2.
\]
If \(\eta\in\Pi(\mu,\nu)\) satisfies \(d_{\mathrm{BL}}(\eta,\bar\pi)\ge r\), then
\[
d_{\mathrm{BL}}(\eta,\pi_\varepsilon)\ge \frac r2.
\]
By Lemma~\ref{lem:strong-stability-universal}, we have
\[
\Delta(\eta)+\varepsilon\Ent(\eta)-\mathcal V(\varepsilon)
=
F_\varepsilon(\eta)-F_\varepsilon(\pi_\varepsilon)
\ge
\frac{\varepsilon}{4}\,d_{\mathrm{BL}}(\eta,\pi_\varepsilon)^2
\ge
\frac{\varepsilon r^2}{16}.
\]
We now take the infimum over all exterior admissible \(\eta\), and obtain
\[
\mathcal V^{\mathrm{out}}_{\bar\pi,r}(\varepsilon)-\mathcal V(\varepsilon)
\ge
\frac{\varepsilon r^2}{16}
\]
for all sufficiently small \(\varepsilon\). If the exterior admissible set is
empty, the inequality is trivial. It follows that
\[
\liminf_{\varepsilon\downarrow0}
\frac{\mathcal V^{\mathrm{out}}_{\bar\pi,r}(\varepsilon)-\mathcal V(\varepsilon)}
{\varepsilon}
\ge
\frac{r^2}{16}>0.
\]

For the converse, we assume \textit{(ii)} and suppose that
\(\pi_\varepsilon\) does not converge weakly to \(\bar\pi\). Given that
\(d_{\mathrm{BL}}\) metrises weak convergence on \(\Pi(\mu,\nu)\), there exist
\(\delta>0\) and a sequence \(\varepsilon_n\downarrow0\) such that
\[
d_{\mathrm{BL}}(\pi_{\varepsilon_n},\bar\pi)\ge \delta
\qquad\text{for all }n.
\]
By weak compactness of \(\Pi(\mu,\nu)\), we may pass to a subsequence and
assume that
\[
\pi_{\varepsilon_n}\Rightarrow\eta
\qquad\text{for some }\eta\in\Pi(\mu,\nu).
\]
Since \(d_{\mathrm{BL}}\) metrises weak convergence, we get
\[
d_{\mathrm{BL}}(\eta,\bar\pi)\ge \delta,
\]
so\(\eta\neq\bar\pi\). We set
\[
r:=\frac12\,d_{\mathrm{BL}}(\eta,\bar\pi)>0.
\]
Since \(d_{\mathrm{BL}}(\pi_{\varepsilon_n},\eta)\to0\), we have
\(d_{\mathrm{BL}}(\pi_{\varepsilon_n},\bar\pi)\ge r\) for all sufficiently
large \(n\). This implies that \(\pi_{\varepsilon_n}\) is admissible for
\(\mathcal V^{\mathrm{out}}_{\bar\pi,r}(\varepsilon_n)\). Because \(\pi_{\varepsilon_n}\) is the unique minimiser of \(\Delta+\varepsilon_n \Ent\), we obtain
\[
\mathcal V^{\mathrm{out}}_{\bar\pi,r}(\varepsilon_n)=\mathcal V(\varepsilon_n)
\qquad\text{for all sufficiently large }n.
\]
But this is a contradiction to \textit{(ii)}, which implies \(\pi_\varepsilon\Rightarrow\bar\pi\).
\end{proof}

\begin{proof}[Proof of Theorem~\ref{thm:selection-criteria}]
Part \textit{(i)} is Proposition~\ref{prop:canonical-local-gap}, and part
\textit{(ii)} is Proposition~\ref{prop:canonical-exterior-gap}.
\end{proof}

\section{Details of the \texorpdfstring{$2\times2$}{2x2} Sinkhorn scaling}\label{sec:appendix-sinkhorn}
In this section, we provide the algebraic details of the diagonal-scaling calculations for the block kernel matrix \(\mathsf K_\varepsilon\) that we used in Proposition~\ref{prop:cantor-pi-eps-explicit}.

\begin{lemma}[Diagonal scaling identities]\label{lem:appendix-sinkhorn-scaling}
We fix \(\varepsilon>0\), and let \(\mathsf K_\varepsilon\) be the block kernel
matrix in \eqref{eq:cantor-kernel-matrix}. Let \(I_0(\varepsilon),I_1(\varepsilon),J_0(\varepsilon),J_1(\varepsilon)\) be as in
Section~\ref{sec:cantor-schrodinger}, and set \(A_\varepsilon,B_\varepsilon,w_\varepsilon\) by \eqref{eq:cantor-AB-def}--\eqref{eq:cantor-what-def}. Let \(u_\varepsilon,v_\varepsilon\) be given by
\eqref{eq:cantor-sinkhorn-constants}, and write
\(U_\varepsilon:=\diag(u_\varepsilon(0),u_\varepsilon(1))\) and
\(V_\varepsilon:=\diag(v_\varepsilon(0),v_\varepsilon(1))\). It follows that the scaled
matrix \(U_\varepsilon\mathsf K_\varepsilon V_\varepsilon\) has unit row sums
and unit column sums. Moreover,
\[
u_\varepsilon(0)v_\varepsilon(0)I_0(\varepsilon)
=
u_\varepsilon(1)v_\varepsilon(1)J_0(\varepsilon)
=
w_\varepsilon,
\]
\[
u_\varepsilon(0)v_\varepsilon(1)I_1(\varepsilon)
=
u_\varepsilon(1)v_\varepsilon(0)J_1(\varepsilon)
=
1-w_\varepsilon.
\]
In particular,
\[
\frac12\Bigl(u_\varepsilon(0)v_\varepsilon(0)I_0(\varepsilon)+u_\varepsilon(1)v_\varepsilon(1)J_0(\varepsilon)\Bigr)=w_\varepsilon.
\]
\end{lemma}

\begin{proof}
We first remove the dependence on \(\varepsilon\). In this case, we use
\(I_0,I_1,J_0,J_1,A,B,w,u,v\) to denote the respective quantities, and
\(\mathsf K,U,V\) denote \(\mathsf K_\varepsilon,U_\varepsilon,V_\varepsilon\).
By definition,
\[
A:=\sqrt{I_0J_0},\qquad B:=\sqrt{I_1J_1},\qquad w:=\frac{A}{A+B},
\]
and
\[
v(0)=1,\qquad v(1)=\sqrt{\frac{I_0J_1}{J_0I_1}},\qquad
u(0)=\frac{w}{I_0},\qquad u(1)=\frac{1-w}{J_1}.
\]

\textit{Row sums.}
We use \(v(1)I_1/I_0=\sqrt{I_1J_1/(I_0J_0)}=B/A\), and get
\[
u(0)\bigl(v(0)I_0+v(1)I_1\bigr)
=w\Bigl(1+\frac{v(1)I_1}{I_0}\Bigr)
=\frac{A}{A+B}\Bigl(1+\frac{B}{A}\Bigr)=1.
\]
Analogously, \(v(1)J_0/J_1=\sqrt{I_0J_0/(I_1J_1)}=A/B\), so
\[
u(1)\bigl(v(0)J_1+v(1)J_0\bigr)
=(1-w)\Bigl(1+\frac{v(1)J_0}{J_1}\Bigr)
=\frac{B}{A+B}\Bigl(1+\frac{A}{B}\Bigr)=1.
\]

\textit{Column sums.}
From \(u(0)I_0=w\) and \(u(1)J_1=1-w\), we have
\[
v(0)\bigl(u(0)I_0+u(1)J_1\bigr)=w+(1-w)=1.
\]
Furthermore,
\[
\begin{aligned}
v(1)u(0)I_1
&=w\,\frac{v(1)I_1}{I_0}
 =w\,\frac{B}{A}
 =1-w,\\
v(1)u(1)J_0
&=(1-w)\,\frac{v(1)J_0}{J_1}
 =(1-w)\,\frac{A}{B}
 =w,
\end{aligned}
\]
so
\[
v(1)\bigl(u(0)I_1+u(1)J_0\bigr)=(1-w)+w=1.
\]
This establishes the row and column sum identities.

It remains to consider the entries. We have \(u(0)v(0)I_0=u(0)I_0=w\), and the
last statement gives \(u(1)v(1)J_0=w\). Similarly, the previous calculation gives
\(u(0)v(1)I_1=1-w\), while \(u(1)v(0)J_1=u(1)J_1=1-w\) by the definition of \(u(1)\). Therefore,
\[
\frac12\Bigl(u(0)v(0)I_0+u(1)v(1)J_0\Bigr)=w,
\]
as required.
\end{proof}

\section{Diagonally \texorpdfstring{$H$}{H}-invariant zero-cost couplings}\label{app:H-invariant-face}
\begin{lemma}[Classification of $H$-invariant zero-cost couplings]\label{lem:cantor-h-classification}
Let \(P\in\Pi(\mu,\mu)\) be invariant under diagonal translations
\[
\tau_s(x,y):=(x+s,y+s),\qquad s\in H,
\]
and assume that \(P\) is supported on the zero-cost set \(Z\). Then there exists
\(w\in[0,1]\) such that
\[
P=w\,P_0+(1-w)\,P_{t_0}.
\]
\end{lemma}

\begin{proof}[Proof of Lemma~\ref{lem:cantor-h-classification}]
We disintegrate $P$ with respect to its first marginal $\mu$. There exists a
$\mu$-a.e.\ unique measurable kernel \(x\mapsto P_x\in\mathcal P(G)\) such that
\[
P(\dd x,\dd y)=\mu(\dd x)\,P_x(\dd y).
\]
Because $P$ is supported on $Z$, we have
\[
P(Z^c)=0
\quad\Longrightarrow\quad
P_x\bigl(Z_x\bigr)=1\ \text{ for }\mu\text{-a.e. }x,
\qquad
Z_x:=\{y:(x,y)\in Z\}=\{x,x+t_0\}.
\]
It follows that, for $\mu$-a.e.\ $x$, there exists a number \(q(x)\in[0,1]\) such that
\begin{equation}\label{eq:cantor-hclass-kernel}
P_x=q(x)\,\delta_x+(1-q(x))\,\delta_{x+t_0}.
\end{equation}
We may choose the function \(q\) to be Borel measurable. Indeed, the diagonal
\(\{(x,y):y=x\}\subset G^2\) is closed, so \(\phi(x,y):=\mathbf 1_{\{y=x\}}\) is Borel, and
\[
q(x)=P_x(\{x\})=\int_G \phi(x,y)\,P_x(\dd y)
\]
is measurable as a parametrised integral.

We begin by establishing the property that $q$ is $H$-invariant in the $\mu$-a.e.\ sense. Let us fix $s\in H$, write \(\tau_s(x,y)=(x+s,y+s)\), \(T_s(z)=z+s\), and also set \(P^s:=(\tau_s)_\#P\). This implies that \((\pr_X)_\#P^s=\mu\), as \(\mu\) is
translation-invariant and \(\pr_X\circ\tau_s=T_s\circ\pr_X\). By direct computation, we obtain the following disintegration of \(P^s\) with respect to \(\mu\):
\[
P^s(\dd x,\dd y)=\mu(\dd x)\,Q_x(\dd y),
\qquad
Q_x:=(T_s)_\#P_{x-s}.
\]
Indeed, for Borel \(A,B\subset G\), we have
\[
\begin{aligned}
P^s(A\times B)
&=P((A-s)\times(B-s))                                      \\
&=\int_{A-s}P_x(B-s)\,\mu(\dd x)                            \\
&=\int_A P_{x-s}(B-s)\,\mu(\dd x)                            \\
&=\int_A Q_x(B)\,\mu(\dd x),
\end{aligned}
\]
where in third equality, we applied the translation invariance of \(\mu\).

By assumption, \(P\) is \(\tau_s\)-invariant, so \(P^s=P\). By \(\mu\)-a.e.\
uniqueness of disintegration, \(Q_x=P_x\) for \(\mu\)-a.e.\ \(x\), that is,
\[
P_x=(T_s)_\#P_{x-s}\quad\text{for }\mu\text{-a.e. }x,
\qquad\text{equivalently}\qquad
P_{x+s}=(T_s)_\#P_x\quad\text{for }\mu\text{-a.e. }x.
\]
Let \(N\subset G\) be a \(\mu\)-null set outside which \eqref{eq:cantor-hclass-kernel} holds. Due to the fact that \(\mu\) is translation-invariant, \(N-s\) is also \(\mu\)-null. Let us intersect the full-measure set on which \(P_{x+s}=(T_s)_\#P_x\) holds with \(N^c\cap(N-s)^c\). Observe that on this set, we may evaluate both kernels by \eqref{eq:cantor-hclass-kernel}. Hence, for
\(\mu\)-a.e. \(x\),
\[
q(x+s)=P_{x+s}(\{x+s\})
=((T_s)_\#P_x)(\{x+s\})
=P_x(\{x\})
=q(x).
\]
It follows that \(q(\cdot+s)=q(\cdot)\) \(\mu\)-a.e. for every fixed \(s\in H\).

We now apply Lemma~\ref{lem:cantor-Haar-averaging} to \(q\) and set
\(\bar q:=\mathsf A q\). This implies that \(\bar q\) is pointwise \(H\)-invariant and
\(\bar q=q\) \(\mu\)-a.e. One useful consequence of this fact is that replacing \(q\) by \(\bar q\) in \eqref{eq:cantor-hclass-kernel} does not change \(P\). Recall the clopen cosets \(G_0,G_1\). If \(x,x'\in G_i\), then \(x'-x\in H\), so by pointwise \(H\)-invariance, we obtain \(\bar q(x')=\bar q(x)\). It follows that \(\bar q\) is constant on each \(G_i\). We will write these constants as \(q_0,q_1\in[0,1]\). We now show that the second marginal necessarily implies $q_0=q_1$.  Indeed, under \eqref{eq:cantor-hclass-kernel}, if \(x_1=0\), then \(y=x\) has \(y_1=0\), while \(y=x+t_0\) has \(y_1=1\). On the other hand, if \(x_1=1\), then \(y=x\) has \(y_1=1\), while \(y=x+t_0\) has \(y_1=0\). We conclude that
\[
P(\{y_1=0\})
=\int_G P_x(\{y_1=0\})\,\mu(\dd x)
=\frac12\,q_0+\frac12\,(1-q_1).
\]
Since \((\pr_Y)_\#P=\mu\) and \(\mu(\{y_1=0\})=\frac12\), this gives precisely
\(q_0=q_1=:w\). It remains to identify $P$. To this end, let \(\varphi\) be bounded measurable on \(G^2\). By \eqref{eq:cantor-hclass-kernel} with \(q\equiv w\) on both cosets, we have
\[
\begin{aligned}
\int_{G^2}\varphi\,\dd P
&=\int_G\Big(w\,\varphi(x,x)+(1-w)\,\varphi(x,x+t_0)\Big)\,\mu(\dd x) \\
&=w\int_G\varphi(x,x)\,\mu(\dd x)
 +(1-w)\int_G\varphi(x,x+t_0)\,\mu(\dd x),
\end{aligned}
\]
This equals \(\int\varphi\,\dd\bigl(wP_0+(1-w)P_{t_0}\bigr)\), which implies
\(P=wP_0+(1-w)P_{t_0}\).
\end{proof}

\bibliographystyle{plain}
\bibliography{counterexample_limit3_refs}

\end{document}